\documentclass[12pt, a4paper]{article}

\usepackage[utf8x]{inputenc}
\usepackage[english]{babel}
\usepackage{amssymb}
\usepackage{amsfonts}
\usepackage{amsmath}
\usepackage{amsthm}
\usepackage{epsfig}
\usepackage{graphics}
\usepackage{dsfont}
\usepackage[usenames, dvipsnames]{xcolor}
\usepackage{verbatim}
\usepackage{latexsym}
\usepackage{setspace}

\usepackage{tikz}
\usepackage{enumerate}
\usepackage{pgfplots}
\usepackage{hyperref}
\usepackage{cancel}

\usepackage{caption} 
\usepackage{mathtools}

\theoremstyle{plain}
\newtheorem{theorem}{Theorem}[section]
\newtheorem{corollary}[theorem]{Corollary}
\newtheorem{lemma}[theorem]{Lemma}
\newtheorem{proposition}[theorem]{Proposition}

\theoremstyle{definition}

\theoremstyle{definition}
\newtheorem{remark}[theorem]{Remark}

\allowdisplaybreaks[4]

\newcommand{\NN}{\mathbb{N}}

\newcommand{\RR}{\mathbb{R}}
\newcommand{\PP}{\mathbb{P}}

\newcommand{\FF}{\mathbb{F}}

\newcommand{\EE}{\mathbb{E}}

\newcommand{\cR}{\mathcal{R}}

\newcommand{\var}{\operatorname{\mathbb{V}ar}}

\newcommand*\diff{\mathop{}\!\mathrm{d}}
\newcommand{\abs}[1]{\left\vert#1\right\vert}

\newcommand{\balpha}{\boldsymbol{\alpha}}

\newcommand{\bT}{\mathbf{T}}
\newcommand{\bone}{\mathbf{1}}

\newcommand{\PH}{\operatorname{PH}}

\usepackage{marginnote}
\begin{document}
	\title{Moments of the ruin time in a Lévy risk model}

	\author{Philipp Lukas Strietzel\thanks{Technische Universit\"at
			Dresden, Institut f\"ur Mathematische Stochastik, Fakultät Mathematik, 01062 Dresden, Germany, \href{mailto:anita.behme@tu-dresden.de}{\texttt{anita.behme@tu-dresden.de}} and \href{mailto:philipp.strietzel@tu-dresden.de}{\texttt{philipp.strietzel@tu-dresden.de}}, phone: +49-351-463-32425, fax:  +49-351-463-37251.}\; and Anita Behme$^\ast$}
	\date{\today}
	\maketitle
	
	\vspace{-1cm}
	\begin{abstract}
		We derive formulas for the moments of the ruin time in a Lévy risk model and use these to determine the asymptotic behavior of the moments of the ruin time as the initial capital tends to infinity. In the special case of the perturbed Cramér-Lundberg model with phase-type or exponentially distributed claims, we explicitly compute the first two moments of the ruin time. All our considerations distinguish between the profitable and the unprofitable setting. 
	\end{abstract}
	

	

	\section{Introduction}\label{S0}
	\setcounter{equation}{0}

Let $X=(X_t)_{t\geq 0}$ be a \emph{perturbed Cramér-Lundberg risk model}, i.e. set  
\begin{equation}\label{eq-perturbedCLmodel}
	X_t:=x+pt  - \sum_{i=1}^{N_t} S_i + \sigma B_t, \quad t\geq 0,
\end{equation} where  $x\geq 0$ is interpreted as \emph{initial capital}, $p>0$ denotes a constant \emph{premium rate}, the Poisson process $(N_t)_{t\geq 0}$ represents the \emph{claim counting process}, and the i.i.d. positive random variables $\{S_i, i\in\NN\}$ are the \emph{claim size variables} and independent of $(N_t)_{t\geq 0}$. The \emph{perturbation} is scaled by $\sigma\geq 0$, and $(B_t)_{t\geq 0}$ denotes a standard Brownian motion that is independent of all other sources of randomness. \\ 
In this article we compute moments of the \emph{ruin time} 
\begin{equation*} 
	\tau_0^- := \inf \left\{ t\geq 0: ~ X_t<0 \right\},
\end{equation*}  or, more precisely, the \emph{ruin time, given that ruin happens,} 
\begin{equation}\label{eq-ruinconditioned}
	(\tau_0^-|\tau_0^-<\infty)
\end{equation}
of $(X_t)_{t\geq 0}$ as in \eqref{eq-perturbedCLmodel}. In the applications considered in this article we will assume the claims $\{S_i,i\in\NN\}$ to have a phase-type distribution.

Note that in the \emph{non-profitable} case, i.e. whenever $\mathbb{E}_0[X_1]\in [-\infty, 0]$, the risk process $X$ enters the negative half-line almost surely and hence $\tau_0^-=(\tau_0^-|\tau_0^-<\infty)$.  Although this setting is hardly studied in the actuarial context, non-profitable risk processes appear e.g. in multivariate risk models where single branches are allowed to become negative, as long as this can be balanced out by other branches, see e.g. \cite{Hult2005}. Likewise, in regime-switching or dividend models it is sometimes reasonable to choose (controlled) risk processes that are unprofitable for certain time frames, see e.g. \cite{Kulenko2008} or \cite{Thonhauser2007}.\\
In the \emph{profitable} setting, where the \emph{net-profit condition} $\mathbb{E}_0[X_1]>0$ holds, the term \emph{ruin time} will be typically used for the conditioned quantity \eqref{eq-ruinconditioned}.  

The ruin time is an extensively studied quantity in the field of actuarial mathematics. In particular, for the classical \emph{Cramér-Lundberg risk model} without perturbation, i.e. for
\begin{equation*}
	X_t:=x+pt - \sum_{i=1}^{N_t} S_i, \quad t\geq 0,
\end{equation*}
various authors have studied the ruin time or the joint distribution of the ruin time, the surplus immediately before ruin, and the deficit at ruin, using different techniques, see e.g. \cite{Gerber1998, Lin1999, Lin2000} to mention just a few. In particular, in \cite{Lin2000} and \cite{EgidioDosReis2000}, recursion formulas for the moments of the ruin time are provided. In \cite{Drekic2003, Drekic2004} the approach of \cite{Lin2000} is taken up. This leads to closed form expressions for the $k$-th moment of the ruin time in the case of exponential claims, see \cite{Drekic2003}, and a Mathematica program that deals with the heavy algebra involved in calculating explicit moments of the ruin time of a general Cramér-Lundberg process provided in \cite{Drekic2004}. For discrete claim size distributions, moments of the ruin time have also been computed in \cite{Picard1998}.\\
Furthermore, e.g. in \cite{Dickson2002, Pitts2008}, approximations of the moments of the ruin time are considered. In \cite{Frostig2004} upper bounds for the expected ruin time are derived using the duality of the Cramér-Lundberg model with a single server queueing system.  The latter work has been extended in \cite{Frostig2012} to a renewal risk model with phase-type distributed claims.

In the context of the perturbed Cramér-Lundberg model or of an even more general L\'evy risk model, where $X$ is chosen to be any spectrally negative L\'evy process, $\tau_0^-$ is typically referred to as \emph{exit time} or \emph{first passage time}. Most results on $\tau_0^-$ and related quantities in this setting are, however, stated in terms of Laplace transforms, see e.g. \cite[Sec. 9.5]{DoneyBuch} for an overview. Recently, in \cite{BehmeStrietzel2021}, we proved necessary and sufficient conditions for finiteness of moments of the ruin time in a Lévy risk model. In the present article, after a brief explanation of some preliminaries in Section \ref{S1a}, in Section \ref{S1b} 
we derive semi-explicit formulas and asymptotics for the moments of the ruin time in a Lévy risk model.
Afterwards, we provide formulas for the first two  moments of the ruin time of the perturbed Cramér-Lundberg model with phase-type distributed claims in Section \ref{S2}, while in the subsequent Section \ref{S2b} we restrict to exponentially distributed claims which allow for even more explicit results. The final Section \ref{S3} contains the proofs of the results presented in Sections \ref{S1b} and \ref{S2}.

	\section{Spectrally negative Lévy processes and scale functions}\label{S1a}
	\setcounter{equation}{0}

The process $X$ defined in \eqref{eq-perturbedCLmodel} is a special case of a spectrally negative Lévy process, i.e. of a càdlàg stochastic process with independent and stationary increments that does not admit positive jumps.  We assume the process $X$ to be defined on a filtered  probability space $(\Omega,\mathcal{F}, \FF, \mathbb{P})$, and as usual we write $\mathbb{P}_x,$ $\mathbb{E}_x$ for the law and expectation of $X$ given $X_0=x$, respectively.\\ Spectrally negative Lévy processes are typically characterized by their so-called \emph{Laplace exponent} $\psi(\theta) := \tfrac{1}{t}\log \mathbb{E}_0\left[e^{\theta X_t}\right]$, which takes the form
\begin{equation} \label{eq-Laplaceexp}
	\psi(\theta) =  c \theta + \frac{1}{2}\sigma^2 \theta ^2 +  \int_{(-\infty, 0)} \left(e^{\theta x}-1-\theta x\mathds{1}_{\{|x|<1\}}\right)\Pi(\diff x), 
\end{equation}
with constants $c\in \RR$, $\sigma^2\geq 0$ and a Lévy measure $\Pi$ satisfying $\int_{(-\infty,0)} (1\wedge x^2)\Pi(\diff x)<\infty$. In the special case of the perturbed Cramér-Lundberg model \eqref{eq-perturbedCLmodel} this can be reduced to
\begin{equation}\label{eq-LaplaceexpPerturbedCL}
	\psi(\theta) =  p \theta + \frac{1}{2}\sigma^2 \theta ^2 + \lambda \int_{(0,\infty)} \left(e^{-\theta x}-1\right)F(\diff x), 
\end{equation} where $F$ denotes the cdf of the claim sizes $\{S_i,i\in\NN\}$ and $\lambda\geq 0$ is the intensity of the claim arrival process $(N_t)_{t\geq 0}$, such that $\Pi(-dx) = \lambda F(dx), x>0,$ and $p=c-\int_{(-1,0)} x \Pi(\diff x)$.

Our formulas for the moments of the ruin time rely on $q$-scale functions of the spectrally negative Lévy process $X$. Recall that for any $q\geq 0$ the \emph{$q$-scale function} $W^{(q)}\colon\mathbb{R}\to[0,\infty)$  of $X$ is the unique function satisfying  
\begin{equation} \label{eq_proof_Definition_scale_by_laplace}
	\int_0^\infty e^{-\beta x}W^{(q)}(x) \diff x = \frac{1}{\psi(\beta)-q},
\end{equation}
for all 
$$\beta >\Phi(q):= \sup\{\theta\geq 0 : ~ \psi(\theta) = q\}, \quad q\geq 0,$$
and such that $W^{(q)}(x)=0$  for $x<0$. \\
Note that  $q\mapsto W^{(q)}(x)$ may be extended analytically to $\mathbb{C}$, which means especially that on every compact subset of $\mathbb{C}$ it is infinitely often differentiable with bounded derivatives. Hence, limits of the type $q\downarrow 0$ for derivatives of $W^{(q)}$ w.r.t. $q$ exist and are finite.

The Laplace exponent's right inverse $q\mapsto \Phi(q)$ is strictly increasing on $[0,\infty)$, infinitely often differentiable on $(0,\infty)$ and such that $$\Phi(0)=0  \Leftrightarrow  \psi'(0+)\geq 0 \quad \text{while} \quad \Phi(0) >0 \Leftrightarrow \psi'(0+)<0,$$
where $\psi'(0+)=\lim_{\theta\downarrow 0} \psi'(\theta)$.
The function $\Phi$ is the well-defined inverse of $\psi(\theta)$ on the interval $[\Phi(0),\infty)$, i.e. 
\begin{equation*}
	\Phi(\psi(\theta)) = \theta \quad \text{ and } \quad  \psi(\Phi(q)) = q, \qquad \forall \theta \in [\Phi(0),\infty), ~ q\geq 0.
\end{equation*}

We refer to \cite{Kyprianou2014} for proofs of the given properties and a more thorough discussion of (spectrally negative) Lévy processes and scale functions. More detailed accounts on scale functions and their numerous applications can be found in \cite{Avram2020} and \cite{kuznetsov2011}.  

	Throughout this article $\partial_q^k f(q,x)$ denotes the $k$-th derivative of a function $f$ with respect to $q$, while $\partial_q:= \partial_q^1$. In case of only one parameter we will usually omit the subscript and also use the standard notation $\partial^k f(x)=f^{(k)}(x)$. 

 \section{On moments of the ruin time}\label{S1b}
 \setcounter{equation}{0}
 
Throughout this section 
we consider a spectrally negative Lévy process $X$ with Laplace exponent $\psi$ given in \eqref{eq-Laplaceexp}.

  We exclude the case that $X$ is a pure drift which implies that  $\PP_x(\tau_0^-<\infty)>0$. If $X$ is of infinite variation, i.e. if $\sigma^2>0$ or if $\int_{(-1,0)} x \Pi(\diff x)=\infty$, it holds $\PP_0(\tau_0^-=0)=1$, and hence in this case we exclude the initial capital $x=0$ from our considerations.\\
By \cite[Thm. 3.1]{BehmeStrietzel2021} the $k$-th moment of the ruin time is finite if and only if one of the following two assumptions holds:
\begin{enumerate}\label{conditions}
	\item $\mathbb{E}_0[X_1]=\psi'(0+)<0$,
	\item $\mathbb{E}_0[X_1]=\psi'(0+)>0$ and $\EE_0[\abs{X_1}^{k+1}]<\infty$.  
\end{enumerate} 
We will thus exclude the case $\EE_0[X_1]=\psi'(0+)=0$ from now on. We also note that in the case of a perturbed Crámer-Lundberg model \eqref{eq-perturbedCLmodel} the assumption $\mathbb{E}_0[\abs{X_1}^{k+1}]<\infty$ is equivalent to $\mathbb{E}[S_1^{k+1}]<\infty$,  cf. \cite[Thm. 3.8]{Kyprianou2014}.

 The following proposition gives a representation of the $k$-th moment of the ruin time in terms of $\Phi$ and the scale function $W^{(q)}$ of the underlying process. Its proof is given in Section \ref{S3a}. Note that the special case $k=1$ of Equation \eqref{eq_generalMoments_convolution} below can also be derived from \cite[Thm. 6.9.A)]{Avram2020}.
 \begin{proposition} \label{Proposition_formulae_to_work_with}
	Let $X$ be a spectrally negative L\'evy process with Laplace exponent $\psi$. For any $x\geq 0$ and any $k\in \NN$ the $k$-th moment of the ruin time is given by 	
\begin{align}
\lefteqn{\mathbb{E}_x[(\tau_0^-)^k|\tau_0^-<\infty] }  \label{eq_general_formula} \\ 
	&=\frac{(-1)^k}{ \mathbb{P}_x(\tau_0^- <\infty)} \lim_{q\downarrow 0}\left(k \int_0^x \partial_q^{k-1} W^{(q)}(y)\diff y  -  \sum_{\ell=0}^k {k\choose \ell}\cdot \eta^{(\ell)}(q)\cdot \left( \partial_q^{k-\ell}W^{(q)}(x)\right)\right),  \nonumber  
\end{align}
where $\eta(q) := \frac{q}{\Phi(q)}$. Moreover, with $(W^{(0)})^{\ast k}(x)$ denoting the $k$-fold convolution of $W^{(0)}(x)$ with itself, for all $x>0$ 
\begin{equation} \label{eq_generalMoments_convolution} 
	\mathbb{E}_x[(\tau_0^-)^k|\tau_0^-<\infty] =\frac{(-1)^k \cdot k!}{ \mathbb{P}_x(\tau_0^- <\infty)} \left( \int_0^x (W^{(0)})^{\ast k}(y)\diff y  -  \sum_{\ell=0}^k  \frac{\eta^{(\ell)}(0+)}{\ell!}   (W^{(0)})^{\ast (k-\ell+1)}(x) \right). 
\end{equation}
In particular, the left-hand side of \eqref{eq_general_formula} and \eqref{eq_generalMoments_convolution} is finite if and only if the right-hand side is finite, respectively.  
\end{proposition}

 The subsequent Lemma provides semi-explicit formulas for the first two moments of the ruin time under the assumptions (i) or (ii). 
Its proof is done by elementary calculus and a short sketch is provided in Section \ref{S3a}.  Using the same type of arguments, Lemma \ref{Theorem_FirstMoments} could be extended to cover higher moments of orders $k=3,4,\ldots$ Due to the complicated and lengthy expressions that arise in this procedure, we refrain from giving any details.

 \begin{lemma} \label{Theorem_FirstMoments}
 	Let $X$ be a spectrally negative L\'evy process with Laplace exponent $\psi$ and let $x\geq 0$.
 	\begin{enumerate}
 		\item Assume $\psi'(0+)<0$, then all moments of $\tau_0^-$ are finite and in particular
 		\begin{align}
 			\mathbb{E}_x[\tau_0^-] &= \frac{1}{\Phi(0)} W^{(0)}(x) - \int_0^x W^{(0)}(y)\diff y, \label{eq_Theorem_FirstMomentUnprof} \\
 			\mathbb{E}[(\tau_0^-)^2] & =  2 \int_0^x \lim_{q\downarrow 0}\partial_q W^{(q)}(y)\diff y - \frac{2}{\Phi(0)} \lim_{q\downarrow 0} \partial_q W^{(q)}(x) + \frac{2 \cdot W^{(0)}(x)}{\Phi(0)^2\cdot \psi'(\Phi(0))}. \label{eq_Theorem_SecondMomentUnprof}
 		\end{align}
 		\item Assume that $\psi'(0+)>0$, and that $\EE_0[X_1^2]<\infty$ or $\EE_0[|X_1|^3]<\infty$, respectively. Then
 		\begin{equation}\label{eq_Theorem_FirstMomentprof}
 			\mathbb{E}_x[\tau_0^-| \tau_0^-<\infty] = \frac{\psi'(0+) \cdot \lim_{q\downarrow 0} (\partial_q W^{(q)}(x)) + \frac{\psi''(0+)}{2\cdot\psi'(0+)} \cdot W^{(0)} (x) - \int_0^x W^{(0)}(y)\diff y}{1-\psi'(0+) \cdot W^{(0)}(x)}, 
 		\end{equation}  and 
 		\begin{align} \label{eq_Theorem_SecondMomentprof}
 			\lefteqn{\mathbb{E}_x[(\tau_0^-)^2| \tau_0^-<\infty] }\\ 
 			&= \frac{1}{1-\psi'(0+)\cdot W^{(0)}(x)} \cdot \bigg(2 \lim_{q\downarrow 0} \int_0^x \partial_q W^{(q)}(y)\diff y - \psi'(0+)\lim_{q\downarrow 0}\partial_q^2 W^{(q)}(x) \nonumber \\ 
 			& \qquad - \frac{\psi''(0+)}{\psi'(0+)} \cdot \lim_{q\downarrow 0}\partial_q W^{(q)}(x)- \left( \frac{ \psi'''(0+)}{3\cdot \psi'(0+)^2} - \frac{\psi''(0+)^2}{2\cdot \psi'(0+)^3} \right)W^{(0)}(x)  \bigg). \nonumber
 		\end{align}
 	\end{enumerate}
 \end{lemma}

 In order to evaluate any of the obtained formulas \eqref{eq_general_formula}-\eqref{eq_Theorem_SecondMomentprof} for a specific Lévy process $X$, it is necessary to have an explicit expression for its scale function. Collections of processes where this is the case can be found e.g. in \cite{Hubalek2010} and \cite{kuznetsov2011}. However, even in case of rather simple scale functions, the computations needed to obtain closed-form expressions for the moments of the ruin time involve serious computational efforts as we are going to see in Section \ref{S2}. 
 
 \begin{remark} Our approach towards the moments of the ruin time relies on the standard idea of differentiating the Laplace transform of the ruin time. As a basis for this, here and in \cite{BehmeStrietzel2021}, we decided to use the general representation of this Laplace transform via scale functions given e.g. in \cite[Thm. 8.1]{Kyprianou2014}. This approach then results in the above formulas that are valid for any spectrally negative Lévy process. While we decided to apply these formulas for special cases in the upcoming two sections, alternatively, one may directly impose additional assumptions on the risk process that lead to special representations of the Laplace transform. Again by differentiation these can then be used to obtain (recursive) formulas for the moments of the ruin time in specific models. As an example for such an approach we mention \cite{Lee2014}, where a general representation of the moments of the ruin time in a Sparre Andersen model with specified claim size distribution is given.\\
 	An alternative approach towards the computation of higher moments of the ruin time that has been used in the setting of the (perturbed) Cramér-Lundberg model can be found e.g. in \cite{Lin1999, Lin2000, Tsai2002}. In these papers the expected discounted penalty (Gerber-Shiu) functions associated to the model are shown to fulfill certain (defective) renewal equations (see also \cite{Li2005}). Solving these and choosing the appropriate penalty function one may then derive moments of the ruin time.
 \end{remark}

 The above formulas can be used to further study the $k$-th moment of the ruin time via its Laplace transform.  In the unprofitable case, this is done in the next theorem. Again, the proof is given in Section \ref{S3a}.

 \begin{theorem}\label{Theorem_Asymptotics_General_Moments}
 	Let $X$ be a spectrally negative Lévy process with Laplace exponent $\psi$ as in \eqref{eq-Laplaceexp} such that $\psi'(0+)\in [-\infty,0)$,  and fix $x\geq 0$. Then $\PP_x(\tau_0^-<\infty)=1$ and for all $k\in \NN$  the Laplace transform of $\mathbb{E}_x[(\tau_0^-)^k]$ is given by
 	\begin{equation}
 		\label{eq_proof_LaplaceTrafo_uk}
 			\int_0^\infty e^{-\beta x} \mathbb{E}_x[(\tau_0^-)^k] \diff x 
=(-1)^k\cdot k! \cdot \left(\frac{1}{\beta} \cdot \frac{1}{\psi(\beta)^k} - \sum_{\ell=1}^k  \frac{1}{\ell!} \eta^{(\ell)}(0+) \cdot \frac{1}{\psi(\beta)^{k-\ell+1}}\right), 
 	\end{equation}
 for all $\beta>0$ and $\eta(q) = \frac{q}{\Phi(q)}$. For $\beta=\Phi(0)>0$ the right-hand side of \eqref{eq_proof_LaplaceTrafo_uk} has to be understood in the limiting sense as
 \begin{equation} \label{eq_proof_LaplaceuinPhi}	\int_0^\infty e^{-\Phi(0)x} \mathbb{E}_x[(\tau_0^-)^k] \diff x  = (-1)^k \lim_{q\downarrow0} \partial^{k} \left(\frac{1}{\Phi(q)}\right). 
 \end{equation} 
 Moreover, for all $k\in\NN$, it holds
 	\begin{equation} \label{eq_Asymptotics_Unprof}
 		\lim_{x\to\infty} \frac{\mathbb{E}_x[(\tau_0^-)^k]}{x^k} 
 		= \frac{1}{|\mathbb{E}_0[X_1]|^k}.
 	\end{equation}
 \end{theorem}

 An analogue result on the asymptotic behavior of the moments of the ruin time in the profitable case does not hold, as we will also see in the following sections. Actually, whenever $\psi'(0+)>0$ we have  $\mathbb{P}_x(\tau_0^-<\infty) = 1-\psi'(0+) W^{(0)}(x)$, cf. \cite[Thm. 8.1]{Kyprianou2014}, and hence the pre-factor in \eqref{eq_general_formula} depends on $x$. Thus we can only consider the Laplace transform of $\mathbb{E}_x[(\tau_0^-)^k\cdot \mathds{1}_{\{\tau_0^-<\infty\}}]$ in this case as done in the following Proposition. Again, the proof is postponed to Section \ref{S3a}.

\begin{proposition}\label{Proposition_LaplaceTransform_profitable}
	Let $X$ be a spectrally negative Lévy process with Laplace exponent $\psi$ as in \eqref{eq-Laplaceexp} such that $\psi'(0+)\in(0,\infty)$, and fix $x\geq 0$. Choose $k\in \NN$ such that   $\mathbb{E}[\abs{X_1}^{k+1}]<\infty$. Then for any $\beta>0$
	\begin{align} \label{eq_Laplace_Transform_Profitable}
	\lefteqn{\int_0^\infty  e^{-\beta x} \mathbb{E}_x[(\tau_0^-)^k\cdot \mathds{1}_{\{\tau_0^-<\infty\}}]\diff x =(-1)^k \cdot   k!  \Bigg(\frac{1}{\beta \psi(\beta)^k}   - \frac{\psi'(0+)}{\psi(\beta)^{k+1}} } \\ &\qquad - \sum_{\ell=1}^k   \frac{1}{\ell!\cdot \psi(\beta)^{k-\ell+1}} \sum_{j=1}^\ell \frac{\psi^{(j+1)}(0+)}{(j+1)} \cdot B_{\ell,j}\left(\Phi'(0+),...,\Phi^{(\ell-j+1)}(0+)\right) \Bigg),\nonumber
	\end{align}
where $B_{\ell,j}$ denote the partial Bell polynomials. Moreover it holds 
\begin{equation}\label{eq_Limit_Laplace_Transform_Profitable}
\lim_{x\to\infty}	\mathbb{E}_x[(\tau_0^-)^k\cdot \mathds{1}_{\{\tau_0^-<\infty\}}] = \lim_{x\to\infty}	\mathbb{E}_x[(\tau_0^-)^k| \tau_0^-<\infty] \cdot  \PP_x(\tau_0^-<\infty)\to 0.
\end{equation}
\end{proposition}

\section{Phase-type distributed claims}\label{S2}
 \setcounter{equation}{0}

 From here onwards, we consider the perturbed Crámer-Lundberg risk model $X$ as in \eqref{eq-perturbedCLmodel}, where - as before - we exclude the trivial case that $X$ is a pure drift, as well as $x=0$ if  $\sigma^2>0$.   

In this section, we assume that $\{S_i,i\in\NN\}$ is a sequence of i.i.d. random variables with \emph{phase-type distribution}, i.e. $S_i\sim \PH_d(\balpha, \bT)$, with the cdf and density of $S_i$ being given by
\begin{equation*}
	F(z) = 1-\balpha e^{\mathbf{T}z}\mathbf{1}, \qquad \text{and} \qquad f(z) = -\balpha e^{\mathbf{T}z} \bT \bone,  
\end{equation*} 
respectively. Here,  $d\in \NN$, $\balpha\in \RR_{\geq 0}^d$ with $\|\balpha\|_1=1$, $\bT\in\RR^{d\times d}$ is an invertible subintensity matrix and $\bone$ is the $d$-dimensional column vector of ones.
It is well known, see e.g. \cite[Thm. 3.1.16 and Cor. 3.1.18]{bladtnielsen}, that in this case
$$\EE[S_i]=\balpha(-\bT)^{-1} \bone, \quad \text{and }\quad \EE[S_i^2]=2 \balpha (-\bT)^{-2} \bone,$$
and that all moments of $S_i$ exist.
Further, the Laplace exponent of the process $X$ in the current setting  follows from \eqref{eq-LaplaceexpPerturbedCL} and \cite[Thm. 3.1.19]{bladtnielsen} to be 
\begin{equation} \label{eq_Ex_PHLevy_LaplaceExponent}
	\psi(\theta) =  p \theta + \frac{1}{2}\theta^2\sigma^2 - \lambda \left(\balpha(\theta \mathbf{I}-\mathbf{T})^{-1} \bT \bone + 1 \right), \quad \theta\geq 0, 
\end{equation}
with the unit matrix $\mathbf{I}\in\RR^{d\times d}$.
Whenever $\psi'(0+)\neq 0$, the $q$-scale functions of $X$ in the current setting are known explicitly. Namely, let \begin{equation*}
	\mathcal{R}_q := \{z\in\mathbb{C}: ~ \psi(z)=q \text{ and } z\neq \Phi(q)\}, \quad q\geq 0,
\end{equation*}
be the set of (possibly complex) $q$-roots of $\psi$. Then all $z\in \cR_q$ have non-positive real part $\operatorname{\mathfrak{Re}}(z)\leq 0$. Further, we write $n=\abs{\mathcal{R}_q}$ for the number of distinct roots in $\cR_q$, denote $\cR_q=\{\phi_i(q), q=1,\ldots, n\}$ and assume w.l.o.g. that 
$\operatorname{\mathfrak{Re}}(\phi_n(q))\leq ... \leq \operatorname{\mathfrak{Re}}(\phi_1(q))\leq 0$.
Then, assuming that the multiplicity of each $z\in \cR_q$ is one, the $q$-scale function of the process \eqref{eq-perturbedCLmodel} with $S_i\sim \PH_d(\balpha,\bT)$ is given by
\begin{equation} \label{eq_Ex_PHLevy_qScale}
	W^{(q)}(x) = \frac{e^{\Phi(q)x}}{\psi'(\Phi(q))}  + \sum_{i=1}^n \frac{e^{\phi_i(q)x}}{\psi'(\phi_i(q))}, \quad x\geq 0, q\geq 0,
\end{equation}
as stated in \cite[Eq. (5)]{Ivanovs2021}, which relies on  \cite[Sec. 5.4]{kuznetsov2011}. For $\psi'(0+)<0$ the above form of the $q$-scale function has also been given in \cite{Egami2014}, while a special case of $\psi'(0+)>0$ can also be found in \cite[Eq. (9)]{Kyprianou2007}. Note that several of these sources also consider the case of multiple roots, but due to the more lengthy form of the resulting scale functions we shall exclude this case in our exposures. \\
The explicit form of the scale functions \eqref{eq_Ex_PHLevy_qScale} allows for an evaluation of the formulas given in Lemma \ref{Theorem_FirstMoments}, which result in explicit expressions for the first two moments of the ruin time for the considered processes as they will be presented in the following two propositions.  
The technical and lengthy proofs of both propositions are sketched in Section \ref{S3b}.

 We start with our result in the unprofitable setting.
\begin{proposition} \label{Proposition_Unprofitable_PH}
	Let $X$ have the Laplace exponent \eqref{eq_Ex_PHLevy_LaplaceExponent} and assume that $\psi'(0+)<0$. Then for all $x\geq 0$ it holds
	\begin{align*}
		\mathbb{E}_x[\tau_0^-] &= -\frac{x}{\psi'(0+)}+ C_1^{\downarrow} +  \epsilon_1(x), \quad \text{and}\\ 
		\mathbb{E}_x[(\tau_0^-)^2] &=\frac{x^2}{\psi'(0+)^2}   - \frac{2}{\psi'(0+)^2}\left( \frac{\psi''(0+)}{\psi'(0+)} + \frac{1}{\Phi(0)}\right) \cdot x + C_2^{\downarrow} + \epsilon_2(x),
	\end{align*}
	with
		\begin{equation*}
		\begin{aligned}
			C_1^{\downarrow}:={}&\frac{1}{\Phi(0)p}\mathds{1}_{\{\sigma^2=0\}} - \epsilon_1(0) = \frac{1}{\Phi(0)\psi'(0+)} +\frac{\psi''(0+)}{2 \psi'(0+)^{2}},  \\ 
			C_2^{\downarrow} :={}&\frac{2}{\Phi(0)^2 \psi'(\Phi(0)) p }\mathds{1}_{\{\sigma^2=0\}}  - \epsilon_2(0)\\
			={}&  \frac{2\psi''(0+)}{\Phi(0)\psi'(0+)^3} + \frac{2}{\Phi(0)^2 \psi'(\Phi(0)) \psi'(0+)} + \frac{3\psi''(0+)^2}{2\psi'(0+)^4} - \frac{2\psi'''(0+)}{3\psi'(0+)^3},
		\end{aligned}
	\end{equation*}	
and 
\begin{align*}
	\epsilon_1(x) := {}&\sum_{i=2}^n e^{\phi_i(0)x} \cdot \frac{1}{\psi'(\phi_i(0))}\left(\frac{1}{\Phi(0)}-\frac{1}{\phi_i(0)}\right), \\ 
	\epsilon_2(x) := {}&  2 \sum_{i=2}^n \left(\frac{e^{\phi_i(0)x}\cdot x}{\psi'(\phi_i(0))^2}\left( \frac{1}{\phi_i(0)} - \frac{1}{\Phi(0)} \right)  + e^{\phi_i(0)x}\cdot   \frac{\psi''(\phi_i(0))}{\psi'(\phi_i(0))^3}\left(  \frac{1}{\Phi(0)}- \frac{1}{\phi_i(0)} \right) \right. \\
	& \left.\qquad + e^{\phi_i(0)x}\cdot \left( \frac{1}{\Phi(0)^2  \psi'(\Phi(0))\psi'(\phi_i(0))} -  		  \frac{1}{\psi'(\phi_i(0))^2\phi_i(0)^2}  \right)\right),
\end{align*}
such that for $x\to \infty$ we have $\epsilon_j(x)\to 0$, $j=1,2,$  exponentially fast, since $\operatorname{\mathfrak{Re}}\phi_i(0)<0$ for all $i=2,...,n$. 
In the case $x=0$ it holds
$$\EE_0[\tau_0^-] = \frac{1}{\Phi(0)p} \mathds{1}_{\{\sigma^2=0\}}, \qquad \text{ and } \qquad \mathbb{E}_0[(\tau_0^-)^2] = \frac{2}{\Phi(0)^2 \psi'(\Phi(0)) }\mathds{1}_{\{\sigma^2=0\}}.$$	
\end{proposition}

In the profitable setting, the first two moments of the ruin time can be expressed as follows. 
\begin{proposition} 
	\label{Proposition_Profitable_PH}
	Let $X$ have the Laplace exponent \eqref{eq_Ex_PHLevy_LaplaceExponent} and assume that $\psi'(0+)>0$. Then for all $x\geq 0$ it holds
	\begin{align*}
		\lefteqn{\mathbb{E}_x[\tau_0^-|\tau_0^-<\infty] 	=  -  \bigg(\sum_{i=1}^n \frac{e^{\phi_i(0+)x}}{\psi'(\phi_i(0+))} \bigg) ^{-1} \cdot\bigg( \sum_{i=1}^n \frac{ e^{\phi_i(0+)x}\cdot  x}{\psi'(\phi_i(0+))^2}  } \\
		& - \sum_{i=1}^n  e^{\phi_i(0+)x} \cdot \left(\frac{\psi''(\phi_i(0+))}{\psi'(\phi_i(0+))^3}  - \frac{\psi''(0+)}{2 \psi'(0+)^2 \psi'(\phi_i(0+))} + \frac{1}{\psi'(0+) \psi'(\phi_i(0+)) \phi_i(0+)} \right)\bigg) ,\nonumber 
		\end{align*}
	and 
	\begin{align*}
		\mathbb{E}_x[(\tau_0^-)^2|\tau_0^-<\infty] 
		&=  \bigg( \sum_{i=1}^n \frac{e^{\phi_i(0+)x}}{\psi'(\phi_i(0+))} \bigg)^{-1} \cdot \bigg( \sum_{i=1}^n e^{\phi_i(0+)x} \left( \frac{ x^2}{\psi'(\phi_i(0+))^3} +   B_i^\uparrow  x  + C_i^\uparrow   \right) \bigg),
	\end{align*}
	with
	\begin{align*}
		B_i^\uparrow& :=  \frac{\psi''(0+)}{\psi'(0+)^2\psi'(\phi_i(0+))^2}  - \frac{2}{\psi'(\phi_i(0+))^2\phi_i(0+)\psi'(0+)} - \frac{3  \psi''(\phi_i(0+))}{\psi'(\phi_i(0+))^4} , \\
		\text{and} \quad 			
		C_i^\uparrow &:= \frac{2}{\psi'(\phi_i(0+))^2\phi_i(0+)^2\psi'(0+)} - \frac{\psi'''(\phi_i(0+))}{\psi'(\phi_i(0+))^4} +  \frac{ \psi'''(0+)}{3 \psi'(\phi_i(0+)) \psi'(0+)^3} \\
		&\qquad - \frac{\psi''(0+)^2}{2\psi'(\phi_i(0+)) \psi'(0+)^4} - \frac{\psi''(\phi_i(0+))}{\psi'(\phi_i(0+))} \cdot B_i^\uparrow.
	\end{align*}
In particular
\begin{align*}
	\mathbb{E}_0[\tau_0^-|\tau_0^-<\infty] &= \left(\frac{1}{p} \mathds{1}_{\{\sigma^2=0\}} - \frac{1}{\psi'(0+)} \right)^{-1}\cdot \frac{\psi''(0+)}{2\psi'(0+)^2} \frac{1}{p} \mathds{1}_{\{\sigma^2=0\}},\\
\text{and} \quad 	\mathbb{E}_0[(\tau_0^-)^2|\tau_0^-<\infty] &= \left(\frac{1}{p} \mathds{1}_{\{\sigma^2=0\}} - \frac{1}{\psi'(0+)} \right)^{-1} \cdot \left( \frac{\psi'''(0+)}{3\psi'(0+)^3} - \frac{\psi''(0+)^2}{2\psi'(0+)^4}\right)
  \frac{1}{p} \mathds{1}_{\{\sigma^2=0\}}.
\end{align*}
\end{proposition}\medskip

We emphasize at this point that the obtained formulas in Propositions  \ref{Proposition_Unprofitable_PH} and \ref{Proposition_Profitable_PH} may  seem complicated at first sight, but they can be evaluated rather easily in concrete cases as long as $n$ is not too big: As the Laplace exponent $\psi$ is given explicitly in \eqref{eq_Ex_PHLevy_LaplaceExponent}, its derivatives and roots can be determined by standard procedures.
In Figure \ref{fig_PhaseTypeComparisonPlotsMean} we provide evaluations of the obtained formulas for Cramér-Lundberg processes with phase-type distributed claims with and without perturbation. The considered claim size distributions in this figure are
\begin{align*}
	S_i^X \sim \PH_2\Big(\left(\begin{smallmatrix} 1 \\ 0 \end{smallmatrix}\right), \left(\begin{smallmatrix} -1 & 0.05 \\ 0.1 & -0.1 \end{smallmatrix}\right)\Big) ,  
	S_i^Y \sim \PH_3\Big(\left(\begin{smallmatrix} 0.03 \\ 0.57 \\ 0.4 \end{smallmatrix}\right),\left( \begin{smallmatrix} -0.07 & 0 & 0 \\ 0 & -2 & 0 \\ 0 & 0 & -0.5 \end{smallmatrix}\right) \Big), 
	S_i^Z &\sim \PH_1\left( 1 ,  \tfrac{1}{1.57895}\right),
\end{align*}
such that $S_i^Y$ has a hyperexponential distribution, while $S_i^Z$ has an exponential distribution. The parameters are chosen such that $\mathbb{E}[S_i^X] = \mathbb{E}[S_i^Z] \approx \mathbb{E}[S_i^Y]\approx 1.5=:\mu$, while  $\operatorname{var}(S_i^X)\approx \operatorname{var}(S_i^Y) > \operatorname{var}(S_i^Z)$. The ruin times of the resulting (perturbed) Cramér-Lundberg processes $(X_t)_{t\geq 0}, (Y_t)_{t\geq 0}$, and $(Z_t)_{t\geq 0}$ are then compared with the ruin time of a Brownian motion $(B_t)_{t\geq 0}$ with drift $p - \lambda\mu$ such that $\EE_0[B_1]\approx \EE_0[X_1], \EE_0[Y_1], \EE_0[Z_1]$.

\begin{figure}[!htb]
	\centering
	\includegraphics[scale=0.5]{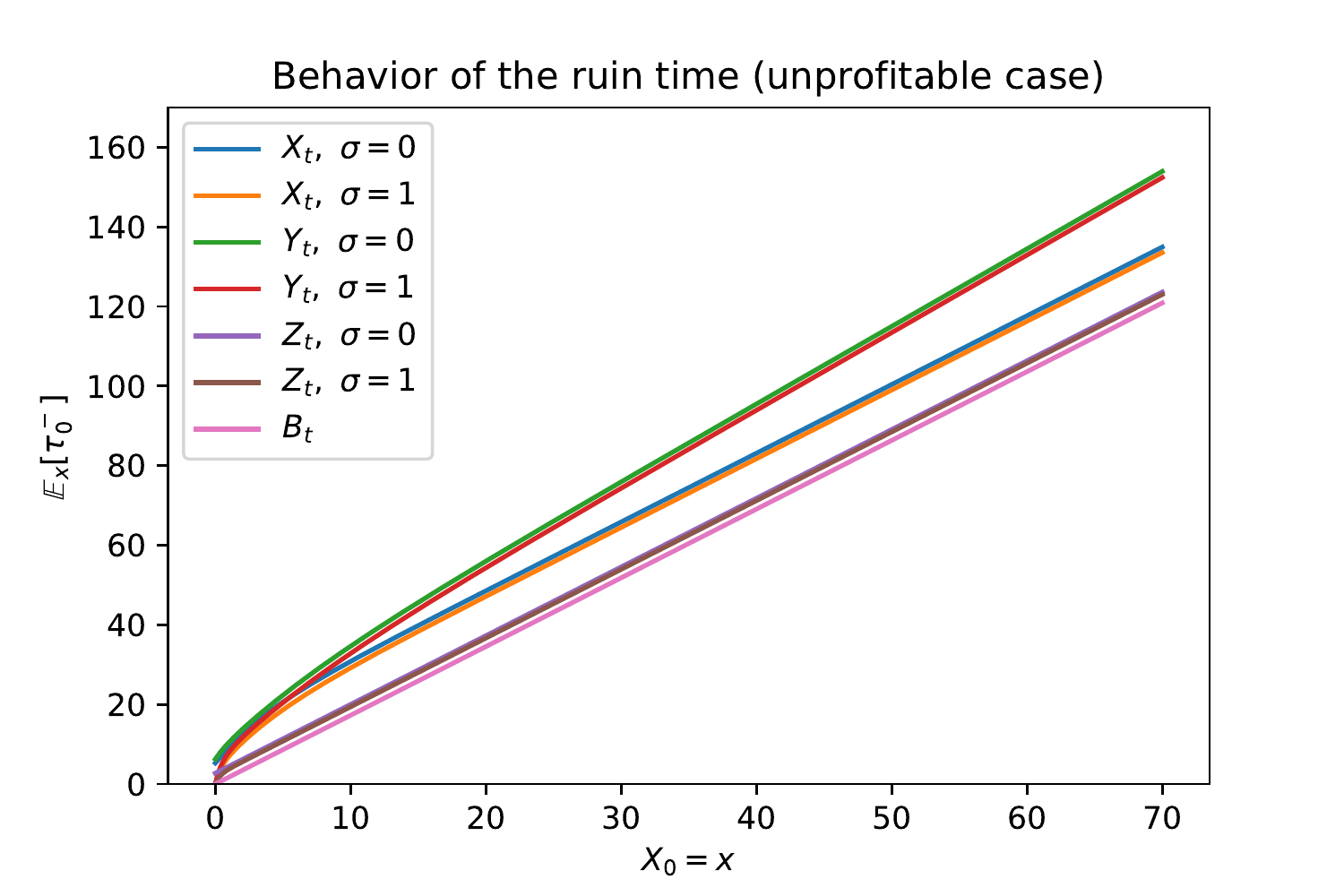}
	\includegraphics[scale=0.5]{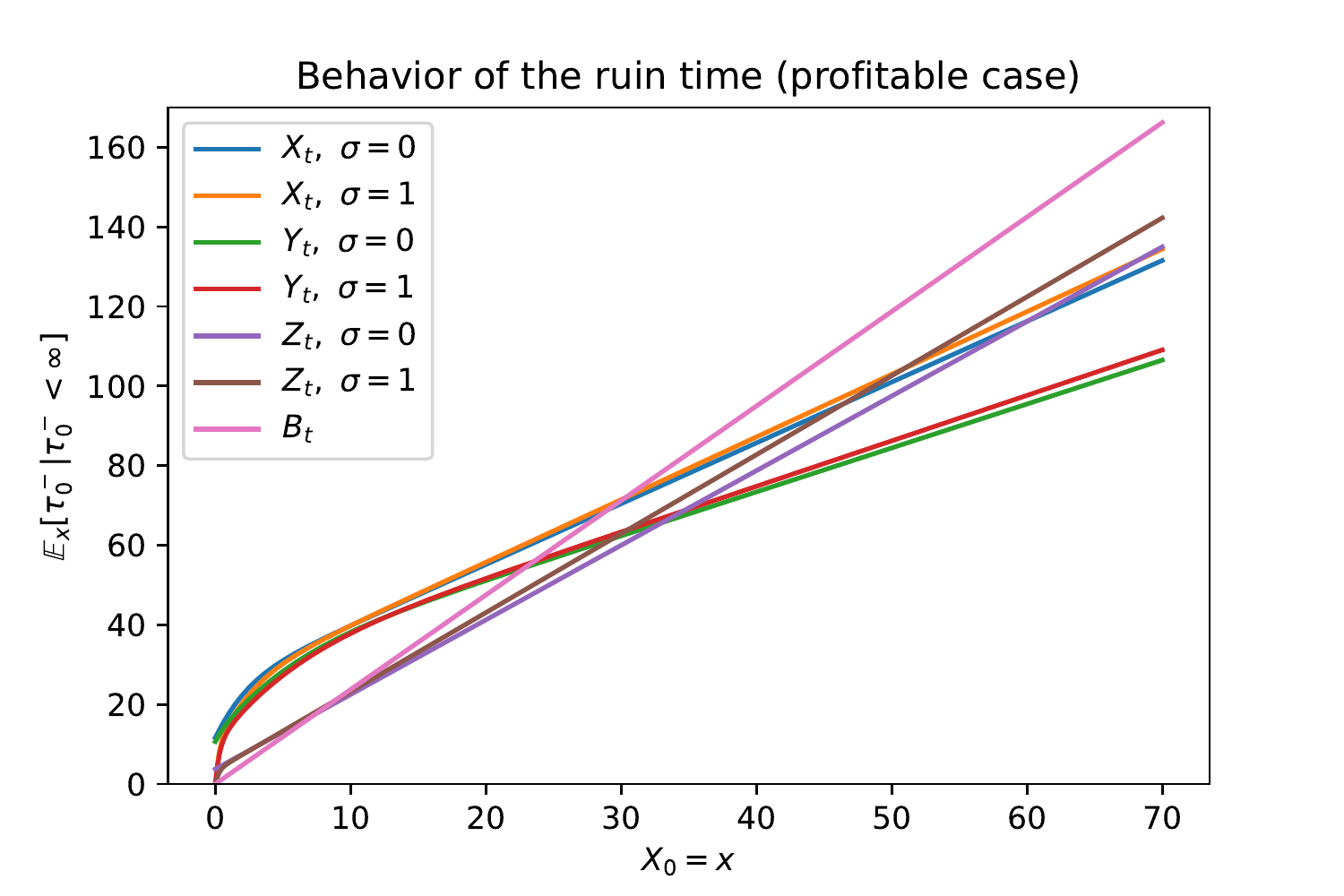}\\
	\includegraphics[scale=0.5]{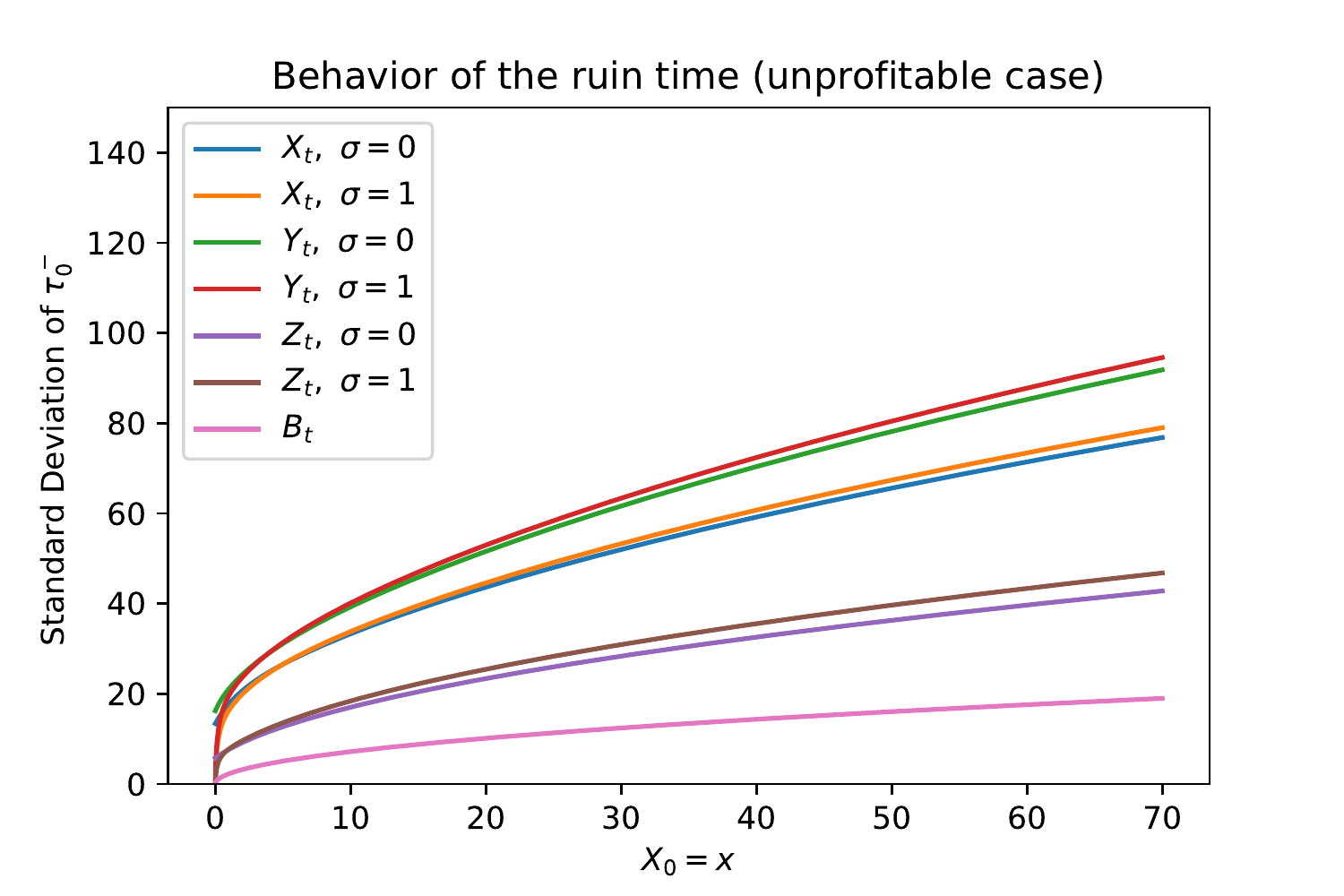}
	\includegraphics[scale=0.5]{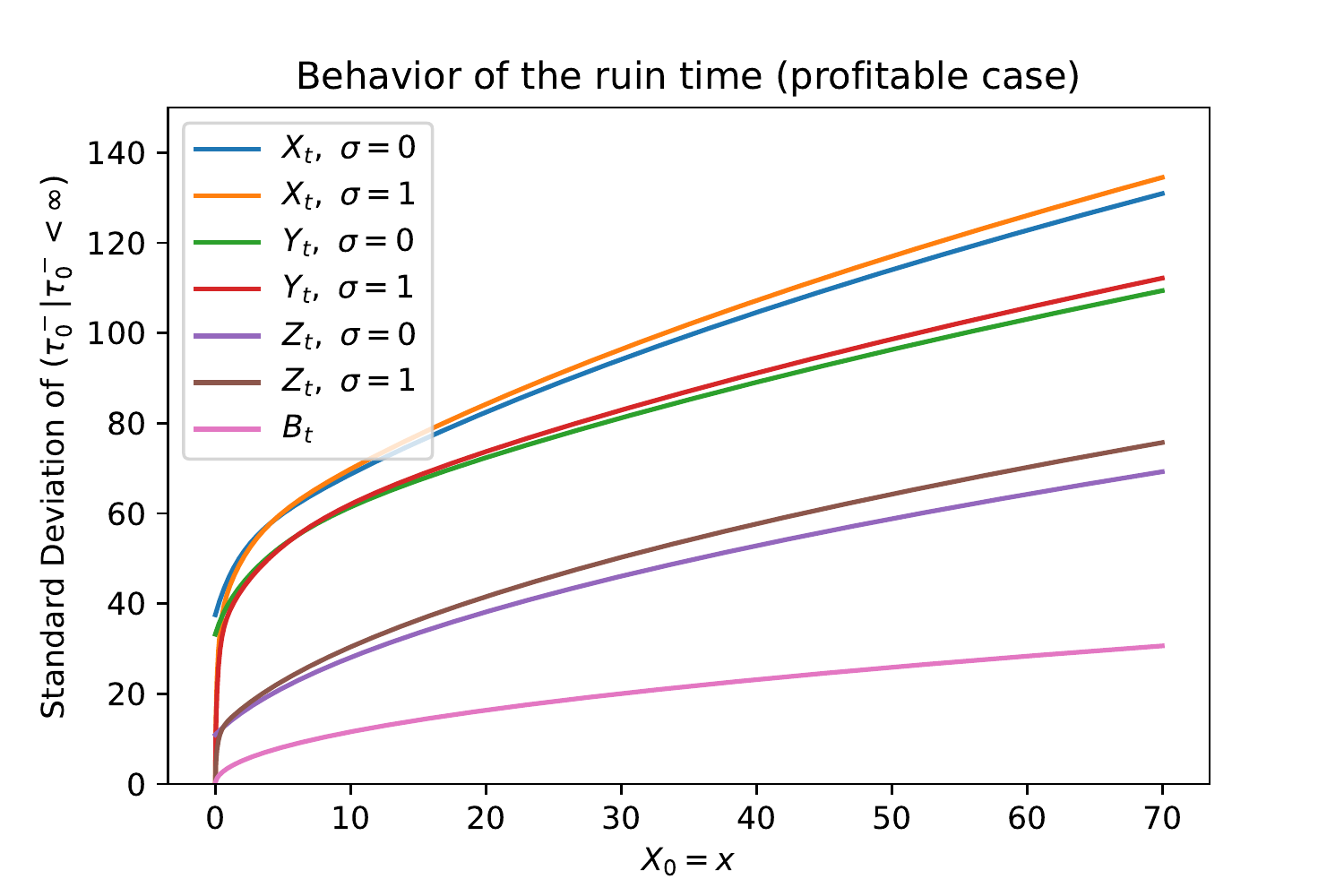}
	\caption{Expected ruin time (given it is finite) on top, and standard deviation of the ruin time (given it is finite) on bottom, in the perturbed Cramér-Lundberg model with different phase-type distributed claims and different choices of $\sigma^2\geq 0$. The parameters chosen are
		$p=1=\lambda$, such that $\psi'(0)\approx -0.5 <0$ (left), and $p=2$, $\lambda=1$, such that $\psi'(0)\approx 0.5 >0$ (right). 
	}
	\label{fig_PhaseTypeComparisonPlotsMean}
\end{figure}

It is obvious from the figure that the behavior in the profitable and the unprofitable regime is crucially different. In the unprofitable setting, in agreement with Theorem \ref{Theorem_Asymptotics_General_Moments}, the principal behavior of the mean ruin time for $x\gg 0$ is solely determined by the expectation of the process. The existence of a perturbation yields a smaller expected ruin time for all $x$. In contrast, in the profitable setting, for large $x$ the perturbed processes have a higher first moment of the ruin time than the unperturbed ones. Further, the pure Brownian motion has the highest expected ruin time but the lowest standard deviation of the ruin time. For small $x$ we observe the same behavior as before: Presence of a perturbation leads to a smaller expectation of the ruin time as the perturbed processes may creep below zero and therefore face ruin very fast.

\begin{remark}
Note that, by the method explained in \cite[Sec. 5.4]{kuznetsov2011}, it is possible to show that a representation for the $q$-scale function as in \eqref{eq_Ex_PHLevy_qScale} exists for any L\'evy process that has a rational transform, i.e., any L\'evy process whose Laplace exponent is a rational function. In the case $\psi'(0+)<0$ this has been done in \cite{Egami2014}. This also allows to extend the results stated above in the context of phase-type distributions to this wider class of processes, since our proofs only rely on the specific representation of the scale function. 
\end{remark}

\section{Exponentially distributed claims}\label{S2b}
 \setcounter{equation}{0}

We again consider the perturbed Cramér-Lundberg model $X$ of the form \eqref{eq-perturbedCLmodel}, where in this section the claim sizes $\{S_i, i\in\NN \}$ are supposed to be i.i.d. exponential random variables with parameter $\gamma>0$. \\
Whenever $\sigma^2=0$ we recover the classical Cramér-Lundberg process with exponential claims, a model that has shown to be extremely well-treatable, as, e.g., ruin probabilities or excess of loss distributions at the time of ruin can be determined in closed form, cf. \cite{asmussenalbrecher}. Also the ruin time in this model has been studied before, e.g. in \cite[Chapter 9.3]{Gerber1979}, where the first moment of the ruin time is computed, or more recently in \cite{Drekic2003}, where explicit formulas for arbitrary integer moments and a pdf of the ruin time are derived. However, all those results are only valid if $\sigma^2=0$ and the net-profit condition $\psi'(0+)=p-\tfrac{\lambda}{\gamma}>0$ is fulfilled, i.e. if the model is profitable with $\psi'(0+)>0$.\\
With the methods derived above, we are now able to present explicit formulas for the first and second moment of the ruin time whenever $\psi'(0+)\neq 0$, and for all choices of $\sigma^2\geq 0$. We thus extend the existing results in two directions by considering a perturbed CL model both in a profitable and in an unprofitable scenario. 

As the exponential distribution is the simplest representative of a phase-type distribution, we can apply the results from the last section. First of all note that the Laplace exponent $\psi$ of $X$ in the current setting is  
\begin{equation} \label{eq_Ex_expLE}
	\psi(\theta) = p \theta + \frac{1}{2} \theta^2 \sigma^2 + \lambda \left( \frac{\gamma}{\theta + \gamma} -1 \right) = p \theta + \frac{1}{2} \theta^2 \sigma^2 -   \frac{\theta \lambda }{\theta + \gamma}, \quad \theta \neq -\gamma ,
\end{equation}
with derivatives 
\begin{equation*}
	\psi'(\theta)=p + \theta \sigma^2  - \frac{\lambda \gamma}{(\theta+\gamma)^{2}} , \quad
	\psi''(\theta)= \sigma^2 +  \frac{2\lambda \gamma}{(\theta+\gamma)^{3}}, \quad \text{and} \quad 
	\psi'''(\theta) = - \frac{6 \lambda \gamma}{(\theta+\gamma)^4}. 
\end{equation*}
The $q$-scale function of $X$ is of the form \eqref{eq_Ex_PHLevy_qScale}, where one can see from \eqref{eq_Ex_expLE} that for every $q>0$ the equation $\psi(\theta)=q$ has exactly three real solutions $\phi_2(q)<-\gamma<\phi_1(q)<0< \Phi(q)$, cf. \cite[Ex. 1.3]{kuznetsov2011}. As $\sigma$ tends to zero, $\phi_2(q)\to -\infty$, and in the limiting case $\sigma^2=0$ only two $q$-roots of $\psi$ exist, which are then easily computed to be, cf. \cite{kuznetsov2011},
\begin{align*}\label{eq_Exexprootssigma0}
	\zeta_{1,2}=\frac{1}{2p} \left(\lambda +q - \gamma p \pm \sqrt{(\lambda +q - \gamma p)^2 +4q\gamma p} \right).
\end{align*}
In particular it follows that
\begin{equation}\label{eq_Exexprootssigma0}
\Phi(0)= \begin{cases}\tfrac{\lambda-p\gamma}{p}, & \text{if }\psi'(0+)<0,\\
0, & \text{if }\psi'(0+)>0, \end{cases} \quad \text{and} \quad \phi_1(0)= \begin{cases} 0, & \text{if }\psi'(0+)<0,\\
- \tfrac{p\gamma-\lambda}{p}, & \text{if }\psi'(0+)>0. \end{cases} 
\end{equation}
For $\sigma^2>0$ no explicit representation of $\Phi(q), \phi_1(q)$ and $\phi_2(q)$, $q>0$, is known. However, for $q=0$, it is easily checked that the three roots are given by
$$\zeta_1=0, \quad \text{and} \quad  \zeta_{2,3}=
\frac{1}{2\sigma^2} \left(-\gamma\sigma^2- 2p \pm \sqrt{\left(\gamma\sigma^2- 2p\right)^2 + 8\lambda \sigma^2}\right),$$
where in the case $\psi'(0+)=p-\tfrac{\lambda}{\gamma} <0$ it holds $\phi_2(0)= \zeta_3 < \phi_1(0)=0 < \Phi(0)=\zeta_2$, while for $\psi'(0+)>0$ we have to reorder and obtain $\phi_2(0) =\zeta_3 < \phi_1(0)= \zeta_2 < \Phi(0)=0$.

The following two corollaries are now immediate consequences of Propositions \ref{Proposition_Unprofitable_PH} and \ref{Proposition_Profitable_PH}, respectively, and can be shown by standard algebra. 

 Again we start with the unprofitable setting.

\begin{corollary} \label{Corollary_Example_ExponentialUnprofitable}
	Let $X$ have the Laplace exponent \eqref{eq_Ex_expLE} with $\sigma^2>0$ and assume that $\psi'(0+)=p-\frac{\lambda}{\gamma}<0$. Set $r:=\sqrt{(\gamma \sigma^2 - 2p)^2 + 8\lambda \sigma^2}$. Then for all $x> 0$ it holds
		\begin{align*} 
		\mathbb{E}_x[\tau_0^-] &=\tfrac{\gamma}{\lambda - p\gamma}\cdot x +C_1^\downarrow \Big(1-  e^{- \frac{\gamma\sigma^2+ 2p + r }{2\sigma^2}  x }\Big) , \quad \text{and}\\
		\mathbb{E}_x[(\tau_0^-)^2]
		&= \tfrac{\gamma^2}{(\lambda -p\gamma)^2}\cdot x^2 + \Big( B^\downarrow - \Big(\tfrac{\gamma \sigma^2 - 2p - r}{\gamma \sigma^2 +2p +r}\Big)^2 \tfrac{4}{r(\lambda-p\gamma)} e^{- \frac{\gamma\sigma^2+ 2p + r }{2\sigma^2}  x}\Big)\cdot x  + C_2^\downarrow \Big(1-  e^{- \frac{\gamma\sigma^2+ 2p + r }{2\sigma^2}  x }\Big) ,
	\end{align*} 
	with 
	\begin{align*}
		B^\downarrow &= \tfrac{2\gamma^2}{(\lambda-p\gamma)^2}\left( \tfrac{2\lambda}{\gamma(\lambda-p\gamma)} + \tfrac{\gamma \sigma^2}{(\lambda-p\gamma)} + \tfrac{2\sigma^2}{\gamma \sigma^2 + 2p - r}\right),\\
		C_1^\downarrow &= \tfrac{ \lambda  }{(\lambda - p\gamma)^2}+ \tfrac{ \gamma^2 \sigma^2 }{2(\lambda - p\gamma)^2} + \tfrac{2 \gamma \sigma^2 }{(\lambda-p\gamma) (\gamma \sigma^2 + 2p - r)} ,\\
		C_2^\downarrow 
		 &= - \tfrac{4\lambda}{(\lambda-p\gamma)^3} + \tfrac{3(\sigma^2 \gamma^2 + 2 \lambda)^2}{2(\lambda-p\gamma)^4} + \tfrac{4\gamma^3 \sigma^4 + 8 \lambda \gamma \sigma^2}{ (\lambda - p\gamma)^3 (\gamma \sigma^2 + 2p- r)} +  \tfrac{16 \gamma \sigma^4}{r(\lambda-p\gamma)} \tfrac{(\gamma \sigma^2 - 2p + r) }{(\gamma \sigma^2 + 2p - r)^3}.
	\end{align*}
\end{corollary}

In the profitable setting, the representations of the first two moments are as follows. 
\begin{corollary} \label{Corollary_Example_ExponentialProfitable}
		Let $X$ have the Laplace exponent \eqref{eq_Ex_expLE} with $\sigma^2>0$ and assume that $\psi'(0+)=p-\frac{\lambda}{\gamma}>0$. Set $r:=\sqrt{(\gamma \sigma^2 - 2p)^2 + 8\lambda \sigma^2}$. Then for all $x> 0$ it holds
		\begin{align*} 
			\mathbb{E}_x[\tau_0^-|\tau_0^-<\infty]
			&= \frac{x\cdot\tfrac{2}{r} \cdot  \tfrac{\gamma \sigma^2 - 2p + r}{ \gamma \sigma^2 + 2p - r } + A^\uparrow_+}{1+ \epsilon^{\operatorname{Exp}}(x)} + \frac{x\cdot \tfrac{2}{r}\cdot  \tfrac{-(\gamma \sigma^2 - 2p - r)}{ \gamma \sigma^2 +2p +r } + A^\uparrow_-}{1+ \epsilon^{\operatorname{Exp}}(x)^{-1}} , \quad \text{and}\\
			\mathbb{E}_x[(\tau_0^-)^2|\tau_0^-<\infty] 
			&= \frac{x^2\cdot \tfrac{4}{r^2}  \Big(\tfrac{\gamma \sigma^2 - 2p + r}{\gamma \sigma^2 + 2p - r }\Big)^2 + x\cdot B_+^\uparrow + C_+^\uparrow }{1+\epsilon^{\operatorname{Exp}}(x)}  + \frac{x^2\cdot \tfrac{4}{r^2} \Big( \tfrac{\gamma \sigma^2 - 2p - r}{ \gamma \sigma^2 +2p +r }\Big)^2+ x\cdot B_-^\uparrow + C_-^\uparrow}{1+\epsilon^{\operatorname{Exp}}(x)^{-1}},
		\end{align*}
	with
		\begin{align*}
		\epsilon^{\operatorname{Exp}}(x) 
		&:=  - \tfrac{\gamma \sigma^2 - 2p - r}{\gamma \sigma^2 - 2p +r}\cdot \tfrac{  \gamma \sigma^2 + 2p - r}{\gamma \sigma^2 +2p +r} \cdot e^{-\frac{r}{\sigma^2} x},
	\end{align*}
and 
	\begin{align*}
		A^\uparrow_{\pm}
		&:= -\tfrac{ \lambda}{(p\gamma-\lambda)^2} - \tfrac{\gamma^2 \sigma^2}{2(p\gamma-\lambda)^2} - \tfrac{2\gamma \sigma^2}{(p\gamma - \lambda)(\gamma \sigma^2+ 2p \mp r)} +\tfrac{4\sigma^2}{r^2}   \Big(\tfrac{\gamma \sigma^2 - 2p \pm r}{ \gamma \sigma^2 + 2p \mp r }\Big)^2 \left[1 + \tfrac{16\lambda \gamma \sigma^4}{(\gamma\sigma^2 - 2p \pm r)^3}\right] , \\
	 B^\uparrow_{\pm}
	 &:=\mp  \tfrac{2}{r} \cdot\tfrac{\gamma \sigma^2 - 2p \pm r}{\gamma \sigma^2 + 2p \mp r }  \left[\tfrac{2 \lambda + \gamma^2 \sigma^2 }{(p\gamma - \lambda)^2} + \tfrac{4\gamma \sigma^2}{(p\gamma - \lambda)(\gamma \sigma^2 + 2p \mp r)} \right]  \pm \tfrac{24\sigma^2}{r^3} \Big( \tfrac{\gamma \sigma^2 - 2p \pm r}{ \gamma \sigma^2 + 2p \mp r }\Big)^3  \left[1 + \tfrac{16 \lambda \gamma \sigma^4}{(\gamma\sigma^2 - 2p \pm r)^3}\right] ,\\
 C^\uparrow_{\pm}
 & :=- \tfrac{2\lambda}{(p\gamma - \lambda)^3} - \tfrac{(2\lambda + \gamma^2\sigma^2)^2}{2(p\gamma - \lambda)^4} \pm \tfrac{96 \gamma\sigma^6}{r^3} \cdot  \tfrac{ \gamma\sigma^2   -2p \mp r}{\left( \gamma \sigma^2 + 2p \mp r \right)^3} \mp \tfrac{16\gamma \sigma^4}{r (p\gamma-\lambda)} \tfrac{ \gamma\sigma^2   -2p \pm r}{\left( \gamma \sigma^2 + 2p \mp r \right)^3}\\
 &\qquad \pm \tfrac{2\sigma^2}{r} \cdot  \tfrac{\gamma \sigma^2 - 2p \pm r}{ \gamma \sigma^2 + 2p \mp r } \left[1 + \tfrac{16\lambda \gamma \sigma^4}{(\gamma \sigma^2 - 2p \pm r)^3} \right]B^\uparrow_{\pm} .
\end{align*} 
\end{corollary}

\begin{remark} Observe that, due to
	$-\frac{r}{\sigma^2}= -\sqrt{\left(\gamma - \frac{2p}{\sigma^2}\right)^2  + \frac{8\lambda}{\sigma^2}}<0,$
	the function $\epsilon^{\operatorname{Exp}}(x)$ in Corollary \ref{Corollary_Example_ExponentialProfitable} above decreases exponentially as $x$ grows. Thus for large $x$ the behavior of the first and second moment of the ruin time is dominated by the nominator of the first summand in the given representations, respectively. In particular, mean and variance of the ruin time are approximately affine linear whenever $x$ is large.
\end{remark}

Inserting the roots \eqref{eq_Exexprootssigma0} in Propositions \ref{Proposition_Unprofitable_PH} and \ref{Proposition_Profitable_PH}, we can also easily derive the first two moments of the ruin time in the classical Cramér-Lundberg model. Note that in the profitable case $\psi'(0+)>0$ these formulas have already been obtained in \cite{Drekic2003}.

\begin{corollary} \label{Corollary_Example_ExponentialUnperturbed}
	Let $(X_t)_{t\geq 0}$ have the Laplace exponent \eqref{eq_Ex_expLE} with $\sigma^2=0$. Then 
	$$\mathbb{E}_x[\tau_0^-|\tau_0^-<\infty]= \begin{cases}
		\frac{\gamma x +1}{\lambda - p \gamma }, & \psi'(0+)=p-\frac{\lambda}{\gamma}<0,\\
		\frac{\frac{\lambda}{p}x + 1}{p\gamma-\lambda},& \psi'(0+)=p-\frac{\lambda}{\gamma}>0,
	\end{cases} \quad x\geq 0,$$
while for $\psi'(0+)\neq 0$ we may summarize to  
$$\var_x(\tau_0^-|\tau_0^-<\infty) = \frac{2\lambda \gamma x}{|p\gamma - \lambda|^3} + \frac{p \gamma + \lambda}{|p \gamma - \lambda|^3}, \quad x\geq 0.$$
	\end{corollary}

Comparing the obtained formulas in Corollaries \ref{Corollary_Example_ExponentialUnprofitable}, \ref{Corollary_Example_ExponentialProfitable} for $\sigma^2>0$ with the results in Corollary \ref{Corollary_Example_ExponentialUnperturbed}, we immediately see that the additional Brownian motion in the perturbed Cram\'er-Lundberg model has a big impact on the ruin time. For small $x$ this is intuitively clear. For large $x$, we note that for $\psi'(0+)<0$ the $\epsilon^{\operatorname{Exp}}_{1,2}$-terms in Corollary \ref{Corollary_Example_ExponentialUnprofitable} vanish and the influence of $\sigma^2$ can only be seen in the appearing constants, while the ascent in $x$ of the expected ruin time is untouched by $\sigma^2$. This behavior coincides with the asymptotics shown in Theorem \ref{Theorem_Asymptotics_General_Moments}. 
In the case $\psi'(0+)>0$ treated in Corollary \ref{Corollary_Example_ExponentialProfitable} also the ascent in $x$ of the expected ruin time does depend on $\sigma^2$ as we already saw in Figure \ref{fig_PhaseTypeComparisonPlotsMean}.

\color{black}

 \section{Proofs}
  \label{S3} \setcounter{equation}{0}
 
 \subsection{Proofs for the results in Section \ref{S1b}}\label{S3a}
 
 \subsubsection*{Proof of Proposition \ref{Proposition_formulae_to_work_with}}
 We use \cite[Lemma 2.1]{BehmeStrietzel2021} for $\kappa=k\in\NN$, in which case the fractional derivative equals a classical derivative. More specifically, we obtain  
 	\[\mathbb{E}_x \left[(\tau_0^-)^k\big|\tau_0^- <\infty \right] = \frac{(-1)^k}{ \mathbb{P}_x(\tau_0^- <\infty)}\cdot \lim_{q\to 0}\partial^k_q \left(Z^{(q)}(x) - \eta(q) \cdot W^{(q)}(x)\right),\]
 	with the integrated scale function $Z^{(q)}(x) = 1+ q \int_0^x W^{(q)}(y) \diff y$. By induction and using the product rule of differentiation  one can show that
 	\begin{align*}
 		\partial^k_q Z^{(q)}(x) &= k \cdot \int_0^x \partial_q^{k-1} W^{(q)}(y)\diff y + q\cdot  \int_0^x \partial_q^{k} W^{(q)}(y)\diff y.
 	\end{align*}
 	Further, an application of the general Leibniz rule yields
 	\begin{align*}
 		\partial_q^k\left( \eta(q) \cdot W^{(q)}(x)\right) &= \sum_{\ell=0}^k {k\choose \ell}\cdot  \eta^{(\ell)}(q) \cdot \left( \partial_q^{k-\ell}W^{(q)}(x)\right),
 	\end{align*}
 	such that we can combine and obtain 
 	\begin{align*}
 		\mathbb{E}_x[(\tau_0^-)^{k}|\tau_0^-<\infty]  
 		&=\frac{(-1)^k}{\mathbb{P}_x(\tau_0^- <\infty) }\cdot  \lim_{q\downarrow 0} \Bigg(k \cdot \int_0^x \partial_q^{k-1} W^{(q)}(y)\diff y + q\cdot  \int_0^x \partial_q^{k} W^{(q)}(y)\diff y \\
 		&\qquad \qquad -  \sum_{l=0}^k {k\choose \ell}\cdot  \eta^{(\ell)}(q)\cdot \left( \partial_q^{k-\ell}W^{(q)}(x)\right)\Bigg).
 	\end{align*}
 	Now \eqref{eq_general_formula} follows, since $W^{(q)}$ and $\int_0^x W^{(q)}(y)\diff y$ are analytical w.r.t. $q$.  To prove \eqref{eq_generalMoments_convolution}, observe that from \cite[Eq. (8.29)]{Kyprianou2014} for $x>0$
 	$$\partial_q^k W^{(q)}(x) = \partial_q^k \sum_{\ell\geq 0} q^\ell (W^{(0)})^{\ast(\ell+1)}(x) = \sum_{\ell \geq k} (\partial_q^k q^\ell) (W^{(0)})^{\ast(\ell+1)}(x),$$
 	which implies $\lim_{q\downarrow 0}\partial_q^k  W^{(q)}(x) = k! (W^{(0)})^{\ast(k+1)}(x)$. \qed
 
 \subsubsection*{Proof of Lemma \ref{Theorem_FirstMoments}} 
 Recall that $\Phi$ is the well-defined inverse of $\psi(\theta)$ on the interval $[\Phi(0),\infty)$. Hence applying the chain rule on $q\mapsto q=\psi(\Phi(q))$ we observe that  
 \begin{equation} \label{Lemma_derivative_inverse}
 	\Phi'(q)= \frac{1}{\psi'(\Phi(q))}, \quad q\geq 0,
 \end{equation}
 where the case $q=0$ is interpreted in the limiting sense $q\downarrow 0$. This implies that (see also \cite[Thm 8.1 (ii)]{Kyprianou2014})
 \begin{equation}\label{eq-limiteta}
 	\lim_{q\downarrow 0} \frac{q}{\Phi(q)} = \begin{cases} \psi'(0+), & \text{if }\psi'(0+)\geq 0 , \\ 0, &\text{else.} \end{cases}
 \end{equation}

 In view of Proposition \ref{Proposition_formulae_to_work_with}, to prove Lemma \ref{Theorem_FirstMoments} one has to compute the first two derivatives of $\eta$ and their behavior as $q\downarrow 0$. For these we obtain by standard calculus, and using \eqref{Lemma_derivative_inverse}, 
 \begin{equation} \label{eq_Eta1(0+)}
 	\eta'(0+)  = 
 	\begin{cases}  \frac{1}{\Phi(0)}, & \text{if } \psi'(0+)<0,  \\ 
 		\frac{\psi''(0+)}{2\cdot \psi'(0+)},  & \text{if } \psi'(0+)>0 \text{ and }\psi''(0+)<\infty, 
 	\end{cases}
 \end{equation}
 while
 \begin{equation} \label{eq_Eta2(0+)}
 	\eta''(0+) = \begin{cases}
 		\frac{-2}{\Phi(0)^2\cdot\psi'(\Phi(0))}, & \text{if }  \psi'(0+)<0,\\
 		\frac{ \psi'''(0+)}{3\cdot\psi'(0+)^2} - \frac{\psi''(0+)^2}{2\cdot \psi'(0+)^3}, & \text{if } \psi'(0+)>0 \text{ and }\psi''(0+), |\psi'''(0+)|<\infty.
 	\end{cases}
 \end{equation}
 An evaluation of \eqref{eq_general_formula} for $k=1$, using \eqref{eq-limiteta} and \cite[Eq. (8.10)]{Kyprianou2014}, yields that
 \begin{align*}
 	\mathbb{E}_x[\tau_0^-|\tau_0^-<\infty]  
 	&= \frac{ (0\vee \psi'(0+)) \cdot \lim_{q\downarrow 0}(\partial_q W^{(q)}(x)) + \eta'(0+) W^{(0)}(x) -\int_0^x W^{(0)}(y)\diff y }{1-(0\vee \psi'(0+)) \cdot W^{(0)} (x)}. 
 \end{align*}
 Via \eqref{eq_Eta1(0+)} we thus derive \eqref{eq_Theorem_FirstMomentUnprof} and \eqref{eq_Theorem_FirstMomentprof}.\\
 Likewise, an evaluation of \eqref{eq_general_formula} for $k=2$ leads to \eqref{eq_Theorem_SecondMomentUnprof} and \eqref{eq_Theorem_SecondMomentprof} via \eqref{eq-limiteta}, \cite[Eq. (8.10)]{Kyprianou2014}, and \eqref{eq_Eta2(0+)}. \qed

 \subsubsection*{Proof of Theorem \ref{Theorem_Asymptotics_General_Moments}}
 
 We first prove Theorem \ref{Theorem_Asymptotics_General_Moments} under an additional condition as formulated in the upcoming proposition. Thereafter, we argue that the additional condition is always fulfilled and hence can be discarded.
 
We abbreviate throughout this section $W:=W^{(0)}$ and set \[u_k(x) := \mathbb{E}_x[(\tau_0^-)^k] =\mathbb{E}_x[(\tau_0^-)^k|\tau_0^-<\infty] , \quad k\in \NN.\] 
 
 \begin{proposition}\label{Prop_Asymptotics_ExtraCond}
 Consider the setting of Theorem \ref{Theorem_Asymptotics_General_Moments} and assume additionally that for some $k \in \mathbb{N}$ there exist $A_k,B_k>0$ such that
 	\begin{equation} \label{eq_Expectation_slower_poly}
 		u_k(x) \leq  A_k + B_k x^k \quad \text{for all }x\geq 0.
 	\end{equation}
 Then the Laplace transform of $u_k$ exists for all $\beta>0$ and it is given by \eqref{eq_proof_LaplaceTrafo_uk} and \eqref{eq_proof_LaplaceuinPhi}. Moreover \eqref{eq_Asymptotics_Unprof} holds for the chosen $k$.
 \end{proposition}
\begin{proof}
To compute the Laplace transform of $u_k$ we recall \eqref{eq_generalMoments_convolution}, where in the current setting $\mathbb{P}_x(\tau_0^-<\infty)=~1$. In \cite{BehmeStrietzel2021} it has been shown that $\psi'(0+)<0$ implies that  $\eta^{(\ell)}(0+)<\infty$ for any $\ell\in\mathbb{N}$. Additionally, $\eta^{(0)}(0+) = \lim_{q\to 0 }q/\Phi(q) = 0$, cf. \cite[Thm. 8.1 (ii)]{Kyprianou2014}. Thus \eqref{eq_generalMoments_convolution} reduces to
\begin{equation} \label{eq_proof_u_representation}
 u_k(x)  = (-1)^k \cdot k! \left( \int_0^x W^{\ast k}(y)\diff y  -  \sum_{\ell=1}^{k}  \frac{\eta^{(\ell)}(0+)}{\ell!}   W^{\ast (k-\ell+1)}(x) \right), \; k\in \NN.
\end{equation}
Recall the definition of $W$ in \eqref{eq_proof_Definition_scale_by_laplace} as the inverse Laplace transform of $1/\psi(\beta)$.
Using standard calculation rules for Laplace transforms, cf. \cite[Chapter XIII]{Feller1971} or  \cite{Widder1946}, we obtain for any $\beta>\Phi(0)$
\begin{align*} 
		\int_0^\infty e^{-\beta x} u_k(x) \diff x  
		&=(-1)^k\cdot k! \cdot \left(\frac{1}{\beta} \cdot \frac{1}{\psi(\beta)^k} - \sum_{\ell=1}^k  \frac{\eta^{(\ell)}(0+)}{\ell!} \cdot \frac{1}{\psi(\beta)^{k-\ell+1}}\right) =: \mathfrak{L}_k(\beta),
\end{align*}
which agrees with the formula in \eqref{eq_proof_LaplaceTrafo_uk}. 
Furthermore, Assumption \eqref{eq_Expectation_slower_poly} implies that the Laplace transform of $u_k$ exists also for $0<\beta\leq \Phi(0)$, cf. \cite[Thm. II.2.1]{Widder1946}. Hence \eqref{eq_proof_LaplaceTrafo_uk} holds for all $\beta>0$ as claimed by \cite[Thm. II.6.3]{Widder1946}.\\
To compute $\mathfrak{L}_k(\Phi(0))$ and prove \eqref{eq_proof_LaplaceuinPhi} we consider the limit $	\lim_{\beta\to\Phi(0)}\mathfrak{L}_k(\beta) $.
Recall that for $q\geq 0$ we have $\psi(\Phi(q))=q$ and that $\Phi$ is a continuous function on $[0,\infty)$. Thus,
\begin{align}
	\lim_{\beta\to\Phi(0)}\left(  \frac{1}{\beta} \cdot \frac{1}{\psi(\beta)^k} - \sum_{\ell=1}^k  \frac{\eta^{(\ell)}(0+)}{\ell!} \cdot \frac{1}{\psi(\beta)^{k-\ell+1}} \right)&= 	\lim_{\gamma\downarrow 0}\left(\frac{1}{\Phi(\gamma)} \cdot \frac{1}{\gamma^k} - \sum_{\ell=1}^k  \frac{\eta^{(\ell)}(0+)}{\ell!} \cdot \frac{1}{\gamma^{k-\ell+1}}\right) \nonumber \\ 
	&= 	\lim_{\gamma\downarrow 0} \frac{\frac{1}{\Phi(\gamma)} - \sum_{\ell=1}^k  \frac{\eta^{(\ell)}(0+)}{\ell!} \cdot \gamma^{\ell-1}}{\gamma^k} \nonumber \\
	&=: \lim_{\gamma \downarrow 0 } g_k(\gamma). \label{eqdefg_k}
\end{align} 
Hereby, as $\eta(q) = \frac{q}{\Phi(q)}$, the general Leibniz rule implies that \begin{align*}
 	\eta^{(k)}(q) 
 	&= \sum_{\ell=0}^k {k \choose \ell}\cdot \left( \partial_q^{\ell} q\right) \cdot \left( \partial_q^{k-\ell}\left(\Phi(q)^{-1}\right)\right) = q\cdot  \partial_q^k\left(\Phi(q)^{-1}\right) +  k\cdot \partial_q^{k-1}\left(\Phi(q)^{-1}\right),
\end{align*} for any $k\geq 1$. Hence we have 
\begin{equation} \label{eq_proof_eta_0}
\eta^{(k)}(0+) = k\cdot \lim_{q\downarrow 0} \partial_q^{k-1}\left(\Phi(q)^{-1}\right), \quad k\in\NN, 
\end{equation} 
which is finite by the results in \cite{BehmeStrietzel2021} as mentioned above. 

We now compute $\lim_{\gamma \downarrow 0 } g_k(\gamma)$, where for $k=1$ with \eqref{eq_proof_eta_0}
\begin{align*}
	\lim_{\gamma\downarrow 0} g_1(\gamma) = \lim_{\gamma\downarrow 0} \frac{\frac{1}{\Phi(\gamma)} -   \frac{1}{\Phi(0)}}{\gamma} =  \lim_{\gamma \downarrow 0} \partial \left(\frac{1}{\Phi(\gamma)}\right)=  \frac{\eta''(0+)}{2}.
\end{align*} 
For $k\geq 2$ we  use l'Hospitals rule to obtain
\begin{align*}
	\lim_{\gamma \downarrow 0 } g_k(\gamma) &= \lim_{\gamma\downarrow 0} \frac{\partial \left(\frac{1}{\Phi(\gamma)}\right)  - \sum_{\ell=2}^k  \frac{\eta^{(\ell)}(0+)}{\ell!} (\ell -1) \cdot \gamma^{\ell-2}}{k\gamma^{k-1}}\\
	& =\lim_{\gamma\downarrow 0} \frac{\partial \left(\frac{1}{\Phi(\gamma)}\right) - \frac{\eta''(0+)}{2}  - \sum_{\ell=3}^k  \frac{\eta^{(\ell)}(0+)}{\ell!} (\ell -1) \cdot \gamma^{\ell-2}}{k\gamma^{k-1}} , \label{eq-prooflimg2}
\end{align*}
which, again in the light of \eqref{eq_proof_eta_0}, implies for $k=2$ 
$$\lim_{\gamma \downarrow 0} g_2(\gamma) = \frac12 \lim_{\gamma \downarrow 0} \frac{\partial \left(\frac{1}{\Phi(\gamma)}\right) - \lim_{\zeta \downarrow 0} \partial \left(\frac{1}{\Phi(\zeta)}\right)}{\gamma} = \frac12 \lim_{\gamma \downarrow 0} \partial^2 \left(\frac{1}{\Phi(\gamma)}\right) = \frac{\eta'''(0+)}{3!}.$$
Further iterating this argument yields
$$\lim_{\gamma\downarrow 0} g_k(\gamma) = \frac{1}{k!} \lim_{\gamma\downarrow 0} \partial^{k}\left(\frac{1}{\Phi(\gamma)}\right)  = \frac{\eta^{(k+1)}(0+)}{(k+1)!}, \quad k\in \NN,$$
 and via \eqref{eqdefg_k} this proves
 $$	\lim_{\beta\to\Phi(0)}\mathfrak{L}_k(\beta) = (-1)^k \cdot k!\cdot \frac{\eta^{(k+1)}(0+)}{(k+1)!} = \frac{(-1)^k}{k+1} \cdot\eta^{(k+1)}(0+) = (-1)^k \cdot \lim_{q\downarrow 0} \partial^{k}\left(\frac{1}{\Phi(q)}\right) $$
 which is \eqref{eq_proof_LaplaceuinPhi}. \\
 To prove the asymptotics as stated in \eqref{eq_Asymptotics_Unprof} we apply a Tauberian theorem:  As $u_k\colon[0,\infty)\to [0,\infty)$ is finite (see  \cite{BehmeStrietzel2021} or Section \ref{S1b}) and monotonely increasing, $U_k(x):= \int_0^x u_k(x) \diff x$ defines an improper distribution function. Moreover, we have
\begin{equation*}
	\lim_{\beta\downarrow 0} \frac{\mathfrak{L}_k(\beta)}{\beta^{-(k+1)}} = (-1)^k\cdot k! \cdot \lim_{\beta\downarrow 0} \left(\frac{\beta^k}{\psi(\beta)^k} - \sum_{\ell=1}^k  \frac{\eta^{(\ell)}(0+)}{\ell!} \cdot \frac{\beta^{k+1}}{\psi(\beta)^{k-\ell+1}}\right) = (-1)^k \cdot k! \cdot\frac{1}{\psi'(0+)^k},
\end{equation*} since $\beta/\psi(\beta)\to \psi'(0+)^{-1}$ by l'Hospital's rule.  Hence,
\begin{equation*}
	\mathfrak{L}_k(\beta) \sim \frac{1}{\beta^{k+1}} \cdot\left((-1)^k \cdot k! \cdot \frac{1}{\psi'(0+)^k} \right), \quad \text{as }\beta \to 0.
\end{equation*}
Now \cite[Thm. XIII.5.4]{Feller1971} yields the claim.
\end{proof}

 It remains to prove that in the setting of Theorem \ref{Theorem_Asymptotics_General_Moments} Assumption \eqref{eq_Expectation_slower_poly} is valid for all $k\in \NN$. We will show this via induction.  The next lemma provides the initial case $k=1$ while the subsequent proposition covers the induction step.

\begin{lemma} Under the assumptions of Theorem \ref{Theorem_Asymptotics_General_Moments} there exist $A,B>0$ such that
	\begin{equation*}
		u_1(x)\leq A+B x \quad \text{for any }x\geq 0. 
	\end{equation*}
\end{lemma}
\begin{proof}
	We use a semi-explicit representation of the scale function as given in  \cite[Eq. (29)]{Avram2020}, which reads 
	\begin{equation*}
		W^{(q)}(x) = \Phi'(q) \cdot \left(e^{\Phi(q)x} -\mathbb{P}_x(T_{\{0\}}<\mathbf{e}_q)\right)=\frac{1}{\psi'(\Phi(q))} \cdot \left(e^{\Phi(q)x} -\mathbb{P}_x(T_{\{0\}}<\mathbf{e}_q)\right), 
	\end{equation*} where $T_{\{0\}}:= \inf\{t\geq 0: ~ X_t=0\}$ denotes the first hitting time of $0$ and $\mathbf{e}_q$ is an exponentially distributed, independent, random time with parameter $q$.  Note that $\Phi'(q) =\frac{1}{\psi'(\Phi(q))}$ can be shown by a simple application of the chain rule; see also \eqref{Lemma_derivative_inverse}.  For $q=0$ we thus have
	\begin{equation} \label{eq_proof_ScaleRepresentation_h}
		W(x) = \frac{1}{\psi'(\Phi(0))} \cdot e^{\Phi(0)x} -\frac{\mathbb{P}_x(T_{\{0\}}<\infty)}{\psi'(\Phi(0))} =: \frac{1}{\psi'(\Phi(0))} \cdot e^{\Phi(0)x} - h(x), 
	\end{equation} where $h(x) \in [0,\psi'(\Phi(0))^{-1}]$ for all $x\geq 0$. Using the explicit form of $u_1$ shown in \eqref{eq_Theorem_FirstMomentUnprof} we now derive
	\begin{align*}
		u_1(x) &= \frac{1}{\Phi(0)} W^{(0)}(x) - \int_0^x W^{(0)}(y)\diff y \\
		&\leq \frac{1}{\Phi(0)} \frac{1}{\psi'(\Phi(0))} \cdot e^{\Phi(0)x} - \int_0^x \left(\frac{1}{\psi'(\Phi(0))} \cdot e^{\Phi(0)y} - \frac{1}{\psi'(\Phi(0))} \right)\diff y \\
		&= \frac{1}{\Phi(0)\cdot\psi'(\Phi(0))} + \frac{1}{\psi'(\Phi(0))} \cdot x, 
	\end{align*}
which yields the claim.
\end{proof}

\begin{proposition}\label{Prop_extracondition}
	Consider the setting of Theorem \ref{Theorem_Asymptotics_General_Moments} and assume \eqref{eq_Expectation_slower_poly} holds for some $k\in \NN$, $A_k, B_k>0$. Then there also exist $A_{k+1}, B_{k+1}>0$ such that 
		\begin{equation*} 
		u_{k+1}(x) \leq  A_{k+1} + B_{k+1} x^{k+1} \quad \text{for all }x\geq 0.
	\end{equation*}
\end{proposition}

\begin{proof}
	Fix $k\in\NN$ such that \eqref{eq_Expectation_slower_poly} holds for some $A_k, B_k>0$ and recall the representation \eqref{eq_proof_u_representation} of $u_k$ used in the proof of Proposition \ref{Prop_Asymptotics_ExtraCond} which yields 
	\begin{align*}
			u_{k+1}(x) 	&=  (-1)^{k+1} (k+1)! \Bigg( \int_0^x W^{\ast (k+1)}(y)\diff y  -  \sum_{\ell=1}^{k}  \frac{\eta^{(\ell)}(0+)}{\ell!}   W^{\ast (k+1-\ell+1)}(x) \Bigg)\\ &\quad  - (-1)^{k+1}\eta^{(k+1)}(0+)\cdot W(x)  \\ 
			&= -(k+1) \cdot  \big(u_k *W\big) (x)  + (-1)^k\cdot \eta^{(k+1)}(0+)\cdot W(x),
		\end{align*}
since
	\begin{align*}
		\int_0^x W^{\ast (k+1)}(y)\diff y & = \int_0^x \int_0^y W^{\ast k}(y-z) W(z) \diff z \,\diff y
		= \left(\int_0^{\cdot} W^{\ast k}(y) \diff y  \ast W\right) (x).
	\end{align*}
Using the representation \eqref{eq_proof_ScaleRepresentation_h} of the scale function $W$ we thus obtain
\begin{align}
	\lefteqn{u_{k+1}(x)} \nonumber\\  &= - (k+1)  \int_0^x  u_k(x-y)  \left(\frac{e^{\Phi(0)y}}{\psi'(\Phi(0))}  - h(y)\right) \diff y + (-1)^k  \eta^{(k+1)}(0+) \left(\frac{e^{\Phi(0)x}}{\psi'(\Phi(0))}   - h(x)\right) \nonumber\\
	&= -(k+1) \int_0^x u_k(x-y) \frac{e^{\Phi(0)y}}{\psi'(\Phi(0))} \diff y  + (k+1) \int_0^x  u_k(x-y) h(y) \diff y \nonumber \\ 
	&\quad + (-1)^k\cdot \eta^{(k+1)}(0+) \left( \frac{\Phi(0)}{\psi'(\Phi(0))}   \left(\int_0^x e^{\Phi(0)y}\diff y +1 \right) -h(x) \right) \nonumber \\ 
	&= \frac{1}{\psi'(\Phi(0))} \int_0^x e^{\Phi(0)y}\Big((-1)^k  \cdot \eta^{(k+1)}(0+) \Phi(0) - (k+1)\cdot u_k(x-y) \Big)\diff y  \nonumber \\ 
	&\quad + (k+1)\int_0^x  u_k(x-y) h(y) \diff y  + (-1)^k \cdot \eta^{(k+1)}(0+)  \left(\frac{\Phi(0)}{\psi'(\Phi(0))}- h(x)\right) \nonumber \\ 
	&=: s_1(x) + s_2(x) + s_3(x).\label{eq_splitu}
\end{align}
Hereby, as $h(x) = \frac{\mathbb{P}_x(T_{\{0\}}<\infty)}{\psi'(\Phi(0))} \in [0,\psi'(\Phi(0))^{-1}]$ for all $x\geq 0$, and as $\eta^{(k+1)}(0+)$ is finite by the results in \cite{BehmeStrietzel2021}, we note that the third summand $s_3(x)$ is bounded by some positive constant.\\
 Moreover, by assumption and boundedness of $h$, we have
\begin{equation*}
	0 \leq s_2(x) \leq \frac{k+1}{\psi'(\Phi(0))} \int_0^x u_k(y)\diff y \leq \frac{k+1}{\psi'(\Phi(0))} \int_0^x (A + B y^k) \diff y \leq A^* + B^* x^{k+1},
\end{equation*} for some constants $A^*,B^*>0$ and any $x\geq 0$, since $\psi'(\Phi(0))>0$ which in turn follows from the strict convexity of $\psi$ and $\psi(0)=\psi(\Phi(0))=0$. 
 It remains to consider 
\begin{align*}
	s_1(x) & = \frac{e^{\Phi(0)x}}{\psi'(\Phi(0))} \int_0^x e^{-\Phi(0)z}\left((-1)^k  \cdot \eta^{(k+1)}(0+)\cdot\Phi(0) - (k+1)\cdot u_k(z) \right)\diff z\\
	&=: \frac{e^{\Phi(0)x}}{\psi'(\Phi(0))} \int_0^x e^{-\Phi(0)z} f(z) \diff z.
\end{align*}
By Proposition \ref{Prop_Asymptotics_ExtraCond}, Equation \eqref{eq_proof_LaplaceuinPhi} holds, which in sight of \eqref{eq_proof_eta_0} implies that  
\begin{align*} \lim_{x\to \infty} \int_0^x e^{-\Phi(0)z} f(z) \diff z & = \lim_{x\to \infty}  \left((-1)^k   \eta^{(k+1)}(0+) (1-e^{-\Phi(0)x}) -  (k+1) \int_0^x e^{-\Phi(0)z} u_k(z) \diff z \right)	\\ 
	&= (-1)^k  \eta^{(k+1)}(0+) - (k+1) \int_0^\infty e^{-\Phi(0)z} u_k(z) \diff z \\
	&= 0. \end{align*} 
Hence we may use l'Hospitals rule to obtain
\begin{align*}
\psi'(\Phi(0))\cdot	\lim_{x\to\infty} \frac{s_1(x)}{x^k} &=   \lim_{x\to\infty} \frac{\int_0^x e^{-\Phi(0)z} f(z) \diff z }{x^k\cdot e^{-\Phi(0)x}}
	=  \lim_{x\to\infty} \frac{f(x)}{k\cdot x^{k-1}- x^k\cdot \Phi(0)} \\
	&= \frac{k+1}{\Phi(0)} \cdot \frac{(-1)^k}{\psi'(0+)^k}, 
\end{align*} as our induction hypothesis implies \eqref{eq_Asymptotics_Unprof} by Proposition \ref{Prop_Asymptotics_ExtraCond}. Hence $s_1,s_2$ and $s_3$ in \eqref{eq_splitu} are bounded polynomially with maximal degree $k+1$. This completes our proof.
\end{proof}

\subsubsection*{Proof of Proposition \ref{Proposition_LaplaceTransform_profitable}}

Throughout this section we abbreviate $W:=W^{(0)}$ as before, set
\[u_k(x) := \mathbb{E}_x[(\tau_0^-)^k \mathds{1}_{\{\tau_0^-<\infty \}}] = \mathbb{P}_x(\tau_0^- <\infty) \mathbb{E}_x[(\tau_0^-)^k|\tau_0^-<\infty] , \quad k\in \NN, \]  and define
\[\varphi(q):= \frac{\psi(q)}{q}, \quad q>0.\]

The proof of Proposition \ref{Proposition_LaplaceTransform_profitable} will be split into two parts. First we prove the following auxiliary result.

\begin{lemma} \label{Lemma_PsiByBeta}
	Assume $k\in\NN$ is such that $\psi^{(k+1)}(0+)$ is finite. Then
	\begin{equation*}
		\lim_{\beta\downarrow 0} \varphi^{(k)}(\beta)  = \frac{\psi^{(k+1)}(0+)}{k+1}. 
	\end{equation*}
\end{lemma}
\begin{proof}
	By the general Leibniz rule 
	\begin{align*}
		\varphi^{(k)}(\beta) 
		&=\sum_{\ell=0}^k {k\choose \ell} \cdot \partial^{k-\ell} (\beta^{-1}) \cdot \partial^\ell (\psi(\beta))  
		=  k!  (-1)^k  \beta^{-k-1}  \sum_{\ell=0}^k (-1)^\ell  \frac{\beta^\ell}{\ell!}  \psi^{(\ell)}(\beta),
	\end{align*}
	where, by assumption, $\psi^{(\ell)}(0+)$ is finite for all $\ell=1,...,k+1$.\\
	Inspired by the Taylor expansion of $\psi$ we define the function $\psi_-\colon [0,\infty) \to\mathbb{R}$ via  
	\begin{align*}
		\psi_-(\beta) :=&  
		\sum_{\ell =0}^{k+1} (-1)^\ell \frac{\beta^\ell}{\ell!}   \psi^{(\ell)}(0+), \qquad \beta\geq 0.
	\end{align*} Clearly, $\psi_-$ is infinitely often differentiable on $(0,\infty)$ with 
\begin{equation} \label{eq_proof_DerivativeHelperPsi-}
		 \psi_-^{(n)} (\beta) = \sum_{\ell=n}^{k+1}(-1)^\ell   \frac{\beta^{\ell-n} }{(\ell-n)!} \psi^{(\ell)}(0+) \underset{\beta \downarrow 0}\longrightarrow (-1)^n \cdot \psi^{(n)}(0+), \quad \text{for }n\leq k+1. 
	\end{equation} 
Moreover, it holds
\begin{align*}
\frac{\beta^{k+1} \varphi^{(k)}(\beta) }{k! \cdot (-1)^k} - \psi_-(\beta) 
	&=  \sum_{\ell=0}^k (-1)^\ell \frac{\beta^\ell}{\ell!}  \left(\psi^{(\ell)}(\beta) - \psi^{(\ell)}(0+)\right) - (-1)^{k+1} \frac{\beta^{k+1}}{(k+1)!} \psi^{(k+1)}(0+),
\end{align*}
	which, after rearrangement, implies
\begin{equation} \label{eq_proof_limitPsiByBeta}
	\lim_{\beta\downarrow 0} \frac{\varphi^{(k)}(\beta) }{k! (-1)^k} =  \lim_{\beta\downarrow 0} \frac{\sum_{\ell=0}^k (-1)^\ell  \frac{\beta^\ell}{\ell!}  \left(\psi^{(\ell)}(\beta) - \psi^{(\ell)}(0+)\right) + \psi_-(\beta) }{\beta^{k+1}} + (-1)^k \frac{\psi^{(k+1)}(0+)}{(k+1)!}.
\end{equation}
Hence, in order to prove the claim, it only remains to show that the limit on the right-hand side of \eqref{eq_proof_limitPsiByBeta} vanishes. Hereby, the nominator $N(\beta)$ in the limit on the right-hand side tends to $0$ as $\beta\downarrow 0$. Moreover, the general Leibniz rule yields for any $n\geq 0$
	\begin{align*}
N^{(n)}(\beta) 
		&=\sum_{\ell=0}^k \frac{(-1)^\ell}{\ell!}  \left( \sum_{j=0}^{n\wedge\ell} {n \choose j} \frac{\ell!}{(\ell-j)!}  \beta^{\ell-j}  \cdot \partial_\beta^{n-j} \left(\psi^{(\ell)}(\beta) - \psi^{(\ell)}(0+)\right) \right) + 	\psi_-^{(n)}(\beta),
	\end{align*}
and thus for $n\leq k+1$, in the light of \eqref{eq_proof_DerivativeHelperPsi-}, we obtain 
\begin{align*}
	\lim_{\beta\downarrow 0} N^{(n)}(\beta)
	&= \sum_{\ell=0}^{n-1} (-1)^\ell    {n \choose \ell}  \psi^{(n)}(0+)   + (-1)^n\cdot \psi^{(n)}(0+) \\
	&= \psi^{(n)}(0+) \cdot \sum_{\ell=0}^{n} (-1)^\ell   {n \choose \ell} = 0,
\end{align*}
 which ensures that we may $(k+1)$-times apply l'Hospital's rule to the limit in \eqref{eq_proof_limitPsiByBeta}. This gives
 \begin{align*}
\lim_{\beta\downarrow 0} \varphi^{(k)}(\beta) &= \frac{(-1)^k}{k+1} \cdot  \lim_{\beta\downarrow 0} N^{(k+1)}(\beta) + \frac{\psi^{(k+1)}(0+)}{k+1} = \frac{\psi^{(k+1)}(0+)}{k+1},
\end{align*}
as stated.
\end{proof}

Now we can compute the Laplace transform of $u_k$.

\begin{proof}[Proof of Proposition \ref{Proposition_LaplaceTransform_profitable}, Equation \eqref{eq_Laplace_Transform_Profitable}]
	Observe that in complete analogy to the proof of Proposition \ref{Prop_Asymptotics_ExtraCond} we get
		\begin{align*}
			u_k(x) &= (-1)^k \cdot k! \left( \int_0^x W^{\ast k}(y)\diff y  -  \sum_{\ell=0}^{k}  \frac{\eta^{(\ell)}(0+)}{\ell!} \cdot  W^{\ast (k-\ell+1)}(x) \right), \; k\in \NN,
	\end{align*}
	with $\eta(0+)>0$, since $\psi'(0+)>0$, while by assumption and the results in \cite{BehmeStrietzel2021} we know that $\eta^{(\ell)}(0+)$ is finite for all $\ell=0,...,k$. Hence,  by the same standard computations as in the proof of Proposition \ref{Prop_Asymptotics_ExtraCond} we obtain 
	\begin{equation} \label{eq_Laplace_Transform_Profitableeta}
	\begin{aligned}
	 \int_0^\infty  e^{-\beta x} u_k(x) \diff x 
		 &=(-1)^k  k!  \left(\frac{1}{\beta \psi(\beta)^k} -  \frac{\psi'(0+)}{\psi(\beta)^{k+1}} - \sum_{\ell=1}^k   \frac{\eta^{(\ell)}(0+)}{\ell! \cdot  \psi(\beta)^{k-\ell+1}}  \right),
		\end{aligned} 
		\end{equation} since $\eta(0+)= \psi'(0+)$.
	 Note that as $\psi'(0+)>0$ implies $\Phi(0)=0$, Equation \eqref{eq_proof_Definition_scale_by_laplace} holds for all $\beta>0$ and this carries over to  \eqref{eq_Laplace_Transform_Profitableeta}.\\
	 As $\varphi(\beta) = \frac{\psi(\beta)}{\beta}$ immediately implies $\varphi(\Phi(q))=  \tfrac{q}{\Phi(q)} =  \eta(q)$, we may now apply Faà di Bruno's formula, cf. \cite[Eq. (2.2)]{Johnson2002},  which yields for $\ell=1,\ldots, k$
	 \begin{align*}
	 	 \eta^{(\ell)} (q) = \partial^\ell \varphi(\Phi(q)) &= \sum_{j=1}^\ell \varphi^{(j)}(\Phi(q)) \cdot B_{\ell,j}\left(\Phi'(q),...,\Phi^{(\ell-j+1)}(q)\right),
	 \end{align*} where $B_{\ell,j}$ denote the partial Bell polynomials. Thus, as $\Phi(q)\downarrow 0$ for $q\downarrow 0$, Lemma \ref{Lemma_PsiByBeta} implies 
	\begin{equation}\label{Lemma_EtaRepresentation} 
		\eta^{(\ell)}(0+) = \sum_{j=1}^\ell \frac{\psi^{(j+1)}(0+)}{(j+1)} \cdot B_{\ell,j}\left(\Phi'(0+),...,\Phi^{(\ell-j+1)}(0+)\right),\quad \ell=1,...,k,
	\end{equation}
and inserting this in \eqref{eq_Laplace_Transform_Profitableeta} gives \eqref{eq_Laplace_Transform_Profitable}.
\end{proof}

\begin{remark}
In \cite{BehmeStrietzel2021} it was shown that in the profitable case for any $\kappa >0$ 
\begin{equation*}
	\mathbb{E}_x[(\tau_0^-)^\kappa | \tau_0^-<\infty] \qquad \Leftrightarrow \qquad  \lim_{q\downarrow 0} \abs{D_q^\kappa\eta(q)}<\infty\qquad \Leftrightarrow \qquad \mathbb{E}_0[\abs{X_1}^{\kappa+1}]<\infty , 
\end{equation*} with $D_q^\kappa$ denoting the $\kappa$-th fractional derivative with respect to $q$. While the proof of the first equivalence turned out to be simple, the given proof of the second equivalence involved various arguments from the toolbox of Lévy processes and subordinators. For all special cases where $\kappa=k\in \NN$ however, the reasoning used above to obtain \eqref{Lemma_EtaRepresentation} provides an alternative proof for this second equivalence in a purely analytical way. 
\end{remark}

In order to prove the remainder of Proposition \ref{Proposition_LaplaceTransform_profitable} we need a preliminary result on the relations between derivatives of $\eta$ and $\psi$ as given in the next Lemma. 

\begin{lemma}
Let $\psi'(0+)\in(0,\infty)$. Choose $k\in\NN$ such that $\psi^{(k+1)}(0+)$ is finite, then 
 \begin{equation}\label{Lemma_DerivativeOfSumWithEtaI}
 	\lim_{q\downarrow 0}  \partial_{q}^n \left(\sum_{\ell=0}^k  \frac{\eta^{(\ell)}(0+)}{\ell!} \cdot \psi(q)^{\ell} \right)= \frac{\psi^{(n+1)}(0+)}{(n+1)}, \quad n=1,\ldots, k.
 \end{equation} 
 \end{lemma}
\begin{proof}
First note that $\eta(\psi(q)) = \frac{\psi(q)}{\Phi(\psi(q))} = \frac{\psi(q)}{q} = \varphi(q)$ and thus, using Faà di Bruno's formula, cf. \cite[Eq. (2.2)]{Johnson2002}, we obtain 
\begin{align*}
	\partial_q^k \varphi(q) &= \sum_{\ell=1}^k \eta^{(\ell)}(\psi(q)) \cdot B_{k,\ell}\left(\psi'(q),...,\psi^{(k-\ell+1)}(q)\right).
\end{align*} 
Since $\psi(0)=0$ and due to Lemma \ref{Lemma_PsiByBeta}, taking the limit $q\downarrow 0$ on both sides of this formula implies 
	\begin{equation} \label{eq_EtaSumming}
	\sum_{\ell=1}^{k} \eta^{(\ell)}(0+) \cdot B_{k,\ell}\left(\psi'(0+),...,\psi^{(k-\ell+1)}(0+)\right) = \frac{\psi^{(k+1)}(0+)}{k+1}.
\end{equation} 
Further, Faà di Bruno's formula yields for any $\ell\geq 1$ and $n=1,\ldots k$ 
\begin{align*}
	\partial_q^n \left(\psi(q)^\ell\right) &= \sum_{j=1}^{n\wedge \ell} \frac{\ell!}{(\ell-j)!} \cdot \psi(q)^{\ell-j} \cdot B_{n,j}\left(\psi'(q),...,\psi^{(n-j+1)}(q)\right) \\ 
	&\to \begin{cases}
		\ell! \cdot B_{n,\ell}\left(\psi'(0+),...,\psi^{(n-\ell+1)}(0+)\right), &n\geq \ell, \\ 
		0, & n<\ell,
	\end{cases}  \qquad \text{ as } q \downarrow 0.
\end{align*}
Now, as $\partial_q^n \left(\psi(q)^0\right)= \partial_q^n 1 = 0$ for any $n\geq 1$, we observe that for any $n= 1,\ldots, k$ 
\begin{equation*}
	\lim_{q\downarrow 0}  \partial_q^n \sum_{\ell=0}^k  \frac{\eta^{(\ell)}(0+)}{\ell!} \cdot \psi(q)^{\ell} = \sum_{\ell=1}^{k\wedge n} \eta^{(\ell)}(0+) \cdot B_{n,\ell}\left(\psi'(0+),...,\psi^{(n-\ell+1)}(0+)\right)
\end{equation*} 
and \eqref{Lemma_DerivativeOfSumWithEtaI} follows immediately from \eqref{eq_EtaSumming}. 
\end{proof}

We are now ready to provide the remainder of the proof of Proposition \ref{Proposition_LaplaceTransform_profitable}. 

\begin{proof}[Proof of Proposition \ref{Proposition_LaplaceTransform_profitable}, Equation \eqref{eq_Limit_Laplace_Transform_Profitable}] Again, we aim to use a standard Tauberian theorem to prove the stated asymptotics. First, observe that by \cite[Lemma 2.3]{kuznetsov2011} the scale function $x\mapsto W(x)$ is almost everywhere differentiable and that left and right derivatives of $W(x)$ exist on $(0,\infty)$. Thus, in view of \eqref{eq_generalMoments_convolution}, it  follows that $u_k'(x)$ also exists almost everywhere and the expression is well-defined. Further, from \eqref{eq_Laplace_Transform_Profitableeta} it follows via integration by parts that
	\begin{align*}
		\beta \cdot (-1)^k  k!  \left(\frac{1}{\beta \psi(\beta)^k}  - \sum_{\ell=0}^k   \frac{1}{\ell!   \psi(\beta)^{k-\ell+1}} \eta^{(\ell)}(0+) \right) &= \beta \int_0^\infty e^{-\beta x} u_k(x) \diff x \\
		&= \int_0^\infty e^{-\beta x} u_k'(x) \diff x + u_k(0+).
	\end{align*}
Hereby, a $k$-fold application of l'Hospitals rule using Lemma \ref{Lemma_PsiByBeta} and \eqref{Lemma_DerivativeOfSumWithEtaI} yields
\begin{align*}
	\lefteqn{\lim_{\beta \downarrow 0} \beta\cdot  (-1)^k   k!  \left(\frac{1}{\beta \psi(\beta)^k} - \sum_{\ell=0}^k   \frac{1}{\ell!  \psi(\beta)^{k-\ell+1}}  \eta^{(\ell)}(0+) \right) }\\
	&= (-1)^k  k!\cdot\lim_{\beta \downarrow 0}\left( \frac{1}{\varphi(\beta)^{k+1}} \cdot \frac{\varphi(\beta) - \sum_{\ell = 0}^k \frac{\eta^{(\ell)}(0+) }{\ell ! } \cdot \psi(\beta)^\ell}{\beta^k}\right)\\
	&= (-1)^k  k! \cdot \frac{1}{\psi'(0+)^{k+1}} \cdot \lim_{\beta \downarrow 0}\frac{\varphi(\beta) - \sum_{\ell = 0}^k \frac{\eta^{(\ell)}(0+) }{\ell ! } \cdot \psi(\beta)^\ell}{\beta^k}\\
	&= 0.
\end{align*}
Thus $\int_{(0,\infty)} e^{-\beta x} u_k'(x) \diff x \to - u_k(0+)$ as $\beta \to 0$ and \cite[Thm. V.4.3]{Widder1946} yields that
$$\lim_{x\to \infty} (u_k(x) - u_k(0+)) = \lim_{x\to \infty} \int_0^x u_k'(y) dy = - u_k(0+),$$
which implies the statement.
\end{proof}

 \subsection{Sketches for the proofs of the results in Section \ref{S2}}\label{S3b}
 The proofs of Propositions \ref{Proposition_Unprofitable_PH} and \ref{Proposition_Profitable_PH} are rather straightforward but lengthy. We thus only sketch some crucial steps of the derivations and explain the general idea. Due to the similarities, we restrict our presentation to the profitable setting $\psi'(0)>0$ as treated in Proposition \ref{Proposition_Profitable_PH}. 
 
First note that by differentiation of \eqref{Lemma_derivative_inverse}  
\begin{align} \label{derivativeinverse-diff1}
	\Phi''(q)&=  - \Psi''(\Phi(q)) \Phi'(q)^3 = - \frac{\psi''(\Phi(q))}{\psi'(\Phi(q))^3} ,\\
\text{and} \quad 	\Phi'''(q) &= 3\cdot \frac{\psi''(\Phi(q))^2}{\psi'(\Phi(q))^5} - \frac{\psi'''(\Phi(q))}{\psi'(\Phi(q))^4}. \label{derivativeinverse-diff2}
\end{align}
Second, using an approach similar to \cite[Section 10.2]{Avram2020}, a partial fraction decomposition yields 
\begin{equation} \label{eq_Helper_NewShortedProof} 
	\frac{1}{\psi(\theta)-q} = \sum_{i=0}^n \frac{1}{\psi'(\phi_i(q)) \cdot (\theta -\phi_i(q))}.
\end{equation}
 Thus, setting $\theta = 0$ and letting $q\downarrow 0$ 
 \begin{equation}
\frac{\psi''(0)}{2 \psi'(0)^2}=\sum_{i=1}^n\frac{1}{\psi'(\phi_i(0))\cdot \phi_i(0)}.  \label{eq_appendix21}
\end{equation} 
Likewise, differentiating \eqref{eq_Helper_NewShortedProof} with respect to $q$ and then letting $q\downarrow 0$ yields 
\begin{equation}
	\frac{3\psi''(0)^2}{2\psi'(0)^4} - \frac{2\psi'''(0)}{3\psi'(0)^3} = 	 \sum_{i=1}^n \left( \frac{2}{\psi'(\phi_i(0))^2\phi_i(0)^2} + \frac{2\psi''(\phi_i(0))}{\psi'(\phi_i(0))^3\phi_i(0)}\right). \label{eq_appendix22}
\end{equation} 
Third, in the perturbed Cramér-Lundberg model we have by \cite[Lemma 3.1]{kuznetsov2011} 
$$W^{(q)}(0)=\frac{1}{p}\cdot \mathds{1}_{\{\sigma^2=0\}}.$$
 Inserting this in \eqref{eq_Ex_PHLevy_qScale} and letting $q\downarrow 0$ (after differentiating once or twice) yields 
 	\begin{align}
 		\frac{1}{p} \mathds{1}_{\{\sigma^2=0\}} - \frac{1}{\psi'(0)}&=	\sum_{i=1}^n\frac{1}{\psi'(\phi_i(0))}, \label{eq_appendixA}\\ 
 		-  \frac{\psi''(0)}{\psi'(0)^3}&= 	\sum_{i=1}^n\frac{\psi''(\phi_i(0))}{\psi'(\phi_i(0))^3}, \label{eq_appendixB}\\
 		\text{and}\quad	\frac{3\psi''(0)^2}{\psi'(0)^5} - \frac{\psi'''(0)}{\psi'(0)^4} &= \sum_{i=1}^n \left(\frac{3 \psi''(\phi_i(0))^2}{\psi'(\phi_i(0))^5} - \frac{\psi'''(\phi_i(0))}{\psi'(\phi_i(0))^4}\right).	\label{eq_appendixC}
 	\end{align}
  Now the remainder of the proof relies on standard calculus: Use the explicit representation of the scale function given in \eqref{eq_Ex_PHLevy_qScale} in order to calculate its primitive (w.r.t. $x$) and derivatives (w.r.t. $q$) as well as their limits for $q\downarrow 0$. Inserting the obtained expressions  in terms of the derivatives and roots of the Laplace exponent $\psi$ into the formulas given in  Lemma \ref{Theorem_FirstMoments} then yields the result after suitable rearrangements via Equations \eqref{derivativeinverse-diff1}, \eqref{derivativeinverse-diff2}, and \eqref{eq_appendix21} - \eqref{eq_appendixC}.

 \section*{Acknowledgements}
 We thank the anonymous referee of an earlier version of this manuscript for his/her comments that helped us to improve this article. 
 
\bibliography{literatureApplicationTTR.bib}

\end{document}